\documentclass[journal]{IEEEtran}

\usepackage{amssymb}
\usepackage{float}
\usepackage{subcaption}
\usepackage{ifthen}
\newboolean{PrintVersion}
\setboolean{PrintVersion}{false}
\usepackage{mathtools}
\usepackage{comment}
\usepackage[utf8]{inputenc}
\usepackage[T1]{fontenc}
\usepackage{array}
\usepackage{tabu}
\usepackage{colortbl}
\usepackage{mathrsfs}
\usepackage{bm}

\usepackage{amsthm}
\usepackage{cleveref}
\newtheorem{lemma}{Lemma}

\newtheorem{theorem}[lemma]{Theorem}
\newtheorem{proposition}[lemma]{Proposition}
\newtheorem{corollary}[lemma]{Corollary}

\newtheorem{definition}[lemma]{Definition}

\newcommand{\mc}[1]{\mathcal{#1}}
\newcommand{\mf}[1]{\mathfrak{#1}}
\newcommand{\norm}[2]{\left\| #1 \right\|_{#2} }
\newcommand{\normm}[1]{\left\| #1 \right\| }

\newcommand{\inn}[2]{\left\langle #1, #2 \right\rangle}

\newcommand{\p}[3]{\frac{\partial^{#3}#1}{\partial#2^{#3}}}

\newcommand{\R}{\mathbb{R}}

\renewcommand{\ss}{{\mathbb{X}}}  % standard \ss is German letter that looks like beta
\newcommand{{\cs}}{{\mathbb{U}}}
\newcommand{{\as}}{{\mathbb{K}}}
\newcommand{{\Yb}}{{\mathbb{Y}}}
\newcommand{\xb}{{\bm{x}}}
\newcommand{\wb}{{\bm{w}}}
\newcommand{\vb}{{\bm{v}}}

\newcommand{\zb}{{\bm{z}}}

\newcommand{\be}{{\bm{e}}}

\newcommand{\pb}{{\bm{p}}}
\newcommand{\qb}{{\bm{q}}}
\newcommand{\ub}{{\bm{u}}}
\newcommand{\rb}{{\bm{r}}}
\newcommand{\A}{{\mathcal{A}}}

\newcommand{\F}{{\mathcal{F}}}
\newcommand{\B}{{\mathcal{B}}}
\newcommand{\Bu}{ {B_{L^p ( 0,\tau; \cs)} (R_1)}} 
\newcommand{\G}{{\mathcal{G}}}

\newcommand{\W}{{\mathbb{W}}}
\newcommand{\V}{{\mathbb{V}}}

\newcommand{\chgs}[1]{{#1}} %% changes and questions from Sajjad

\begin{document}

\title{Optimal Controller and Actuator Design for Nonlinear Parabolic Systems}

\author{{M.~Sajjad~Edalatzadeh, ~\IEEEmembership{Member,~IEEE,} and Kirsten~A.~Morris,~\IEEEmembership{Senior Member,~IEEE,}}
\thanks{M. Sajjad Edalatzadeh is with Department of Applied Mathematics, University of Waterloo, Waterloo, Canada e-mail: msedalatzadeh@uwaterloo.ca}
\thanks{Kirsten A. Morris is with Faculty of Mathematics, University of Waterloo, Waterloo, Canada e-mail: kmorris@uwaterloo.ca}
\thanks{Manuscript received }}
\maketitle

\begin{abstract}
Many physical systems are modeled by nonlinear parabolic differential equations, such as the Kuramoto-Sivashinsky (KS) equation. In this paper, the existence of a concurrent optimal controller and actuator design is established  for semilinear systems.  Optimality equations are provided.   The results are shown to apply to optimal controller/actuator design for the Kuramoto-Sivashinsky equation and also nonlinear diffusion.
\end{abstract}

\section{Introduction}
The best actuator design can improve performance and significantly reduce the cost of the control in  distributed parameter systems; see  for example \cite{morris2015comparison}. The optimal actuator design problem of linear systems has been reviewed in various contexts, see \cite{frecker2003recent,van2001review}. For linear partial differential equations (PDEs), the existence of an optimal actuator location has been proven in the literature. In \cite{morris2011linear}, it is proven that an optimal actuator location exists for a linear system with quadratic cost function if the input operator is compact and continuously depends on actuator locations. Further conditions on operators and cost functions are needed to guarantee the convergence in numerical schemes \cite{morris2011linear}.   Similar results have been obtained for $H_2$ and $H_\infty$ controller design objectives \cite{DM2013,kasinathan2013h}.
 
Nonlinearities can  have a significant effect on dynamics, and such systems cannot be accurately modelled by linear differential equations. Control  of systems modelled by nonlinear partial differential equations (PDE's) has been studied  for a number of applications, including  wastewater treatment systems  \cite{martinez2000},  steel cooling plants \cite{unger2001},  oil extraction through a reservoir \cite{li2003}, solidification models in metallic alloys \cite{boldrini2009}, thermistors \cite{homberg2010optimal}, Schl\"ogl model \cite{buchholz2013,casas2013}, FitzHugh–Nagumo system \cite{casas2013}, micro-beam model \cite{edalatzadeh2016boundary}, static elastoplasticity \cite{reyes2016}, type-II superconductivity \cite{yousept2017optimal}, Fokker-Planck equation \cite{fleig2017},  Schr\"odinger equation with bilinear control \cite{ciaramella2016}, Cahn-Hilliard-Navier-Stokes system \cite{hintermuller2017optimal}, wine fermentation process \cite{merger2017optimal}, time-dependent Kohn-Sham model \cite{sprengel2018investigation}, elastic crane-trolley-load system \cite{kimmerle2018optimal}, and railway track model \cite{edalatzadeh2019stability}.
A review of PDE-constrained optimization theory can be found in the books \cite{hinze2008optimization,leugering2012constrained,troltzsch2010optimal}.
State-constrained optimal control of PDEs has also been studied. In \cite{bergounioux2003structure}, the authors investigated the structure of Lagrange multipliers for state constrained optimal control problem of linear elliptic PDEs. Research on optimal control of PDEs, such as \cite{casas1997pontryagin,raymond1999hamiltonian}, has focused on partial differential equations with certain structures.
Optimal control of differential equations in abstract spaces has rarely been discussed \cite{meyer2017optimal}. This paper extends previous results to abstract differential equations without an assumption of stability.

Few studies have discussed optimal control for general classes of nonlinear distributed parameter systems; and even less have looked into actuator design problem of such systems. Using  a finite dimensional approximation of the original partial differential equation model, optimal actuator location has been addressed for some applications. Antoniades and Christofides \cite{antoniades2001integrating} investigated the optimal actuator and sensor location problem for a transport-reaction process using  a finite-dimensional model. Similarly, Lou and Christofides \cite{lou2003optimal} studied the optimal actuator and sensor location of Kuramoto-Sivashinsky equation using a finite-dimensional approximation. Other research concerned with optimal actuator location for nonlinear distributed parameter systems can be found in \cite{armaou2008robust,moon2006finite,saviz2015optimal}. To our knowledge, there are no theoretical results on optimal actuator design of nonlinear distributed parameter systems.

The results of this paper apply to the Kuramoto-Sivashinsky (KS) equation. This equation was derived by Kuramoto to model angular phase turbulence in reaction–diffusion systems \cite{kuramoto1975formation}, and by Sivashinsky for modeling plane flame propagation \cite{sivashinsky1977nonlinear}. It also models film layer flow on an inclined plane \cite{craster2009dynamics},  directional solidification of dilute
binary alloys \cite{novick1987interfacial}, growth and saturation of the potential of dissipative trapped-ion \cite{laquey1975}, and terrace edge evolution during step-flow growth \cite{bena1993}. From system theoretic perspective, Christofides and Armaou studied the global stabilization of KS equation using distributed output feedback control \cite{christofides2000}. Lou and Christofides investigated the optimal actuator/sensor placement for control of KS equation by approximating the model with a finite dimensional system \cite{lou2003optimal}.  Gomes et al. also studied the actuator placement problem for KS equation using numerical algorithms \cite{gomes2017}. The feedback control as well as optimal actuator arrangement of multidimensional KS equation has been studied in \cite{tomlin2019point}. Controllability of KS equation has also been studied \cite{cerpa2010,cerpa2011}. Optimal control of KS equation using maximum principle was studied in \cite{sun2010}. Optimal control of KS equation with point-wise state and mixed control-state constraints was studied in \cite{gao2016}. Liu and Krstic studied boundary control of KS equation in \cite{liu2001stability}. Al Jamal and Morris studied the relationship between stability and stabilization of linearized and nonlinear KS equation \cite{al2018linearized}.

The paper is organized as follows. Section 2 is a short section containing notation and definitions. Section 3 discusses the existence of an optimal input together with an optimal actuator design to nonlinear parabolic systems. In section 4, the worst initial condition is discussed. In section 5 and 6, Kuramoto-Sivashinsky equation and nonlinear heat equation are discussed, respectively. 
\section{Notation and Definitions}
%%%%%%%%%%
Let $\ss$ be a reflexive Banach space. The notation $\ss_1\hookrightarrow \ss_2$ means that the space $\ss_1$ is  densely and continuously embedded in $\ss_2$. 
Also, letting $I\subset \R$ to be a possibly unbounded interval, the Banach space $C^s(I;\ss)$ consists of all H\"older continuous $\ss$-valued functions with exponent $s$ equipped with norm
\begin{equation}
\norm{\xb}{C^s(I;\ss)}=\norm{\xb}{C(I;\ss)}+\sup_{t,s\in I}\frac{\normm{\xb(t)-\xb(s)}}{|t-s|^s}.
\end{equation}
The Banach space $c^{s}(I,\ss)$ is the space of little-H\"older continuous functions with exponent $s$ defined as all $\xb\in C^s(I;\ss)$ such that
\begin{equation}
\lim_{\delta\to 0} \sup_{t,s\in I, |t-s|\le \delta}\frac{\normm{\xb(t)-\xb(s)}}{|t-s|^s}=0.
\end{equation}
 Also, $W^{m,p}(I;\ss)$ is the space of all strongly measurable functions $\xb:I\to \ss$ for which $\norm{\xb(t)}{\ss}$ is in $W^{m,p}(I,\R)$. 
 For simplicity of notation, when $I$ is an interval, the corresponding space will be indicated  without the braces; for example $C([0,\tau];\ss)$  will be indicated by  $C(0,\tau;\ss) . $ 

Let $\A$ be the generator of an analytic semigroup $e^{\A t}$ on $\ss$. For every $p\in [1,\infty]$ and $\alpha\in (0,1)$, the interpolation space $D_{\A}(\alpha,p)$ is defined as the set of  all $\xb_0 \in \ss$ such that the function 
\begin{equation}
 t \mapsto v(t)\coloneqq\normm{t^{1-\alpha-1/p}\A e^{t\A}\xb_0}
\end{equation}
belongs to $L^p(0,1)$ \cite[Section 2.2.1]{lunardi2012analytic}. The norm on this space is 
$$\norm{\xb_0 }{D_{\A}(\alpha,p)}=\normm{\xb_0 }+\norm{v}{L^p(0,1)}.$$

The Banach space $\W(0,\tau)$ is the set of all $ \xb (\cdot )  \in W^{1,p}(0,\tau;\ss)\cap L^p(0,\tau;D(\A))$ with norm \cite[Section II.2]{bensoussan2015book}
\begin{equation}\notag
\norm{\xb}{\W(0,\tau)}=\norm{\dot{\xb}}{L^p(0,\tau;\ss)}+\norm{\A \xb}{L^p(0,\tau;\ss)}. % \quad \forall \xb\in \W(0,\tau).
\end{equation}

\begin{definition}\label{def-maximal}
The operator $\mc{A}:D(\A)\to \ss$ is said to have maximal $L^p$ regularity if for every $\bm{f}\in L^p(0,\tau;\ss)$, $1<p<\infty$, the equation
\begin{equation}\label{eq-linear}
\left\{\begin{array}{l}
\dot{\xb}(t)=\A\xb(t)+\bm{f}(t),\quad t>0,\\
\xb(0)=\xb_0,
\end{array}\right.
\end{equation} 
admits a unique solution in $\W(0,\tau)$ that satisfies (\ref{eq-linear}) almost everywhere on $[0,\tau]$.
\end{definition}
Every generator of an analytic semigroup on a Hilbert space has maximal $L^p$ regularity
 \cite[Theorem 4.1]{dore1993p}.
 
\section{Nonlinear Parabolic Systems}\label{sec-main results}
Let $\xb(t)$ and $\ub(t)$ be the state and input taking values in reflexive Banach spaces $\ss$ and $\cs$, respectively. Also, let $\rb$ denote the actuator design parameter that takes value in a compact set $K_{ad}$ of a topological space $\as$. Consider the following initial value problem (IVP):
\begin{equation}\label{eq-IVP}
\begin{cases}
\dot{\xb}(t)=\mc{A}\xb(t)+\mc{F}(\xb(t))+\mc{B}(\rb)\ub(t),\quad t>0,\\
 \xb(0)=\xb_0.
\end{cases}
\end{equation}
The linear operator $\A:D(\A)\to \ss$ is assumed to have maximal $L^p$ regularity. In particular, if $\A$ is associated with a sesquilinear form that is bounded and coercive with respect to $\V \hookrightarrow \ss, $ it generates an analytic semigroup on $\ss$ \cite[Lemma 36.5 and Theorem 36.6]{sell2013dynamics}.
%and also $0\notin \sigma(\A)$, the spectrum of $\A$. 

The nonlinear operator $\F(\cdot)$ maps a reflexive Banach space $\V$ to $\ss$ where $D_{\A}(1/p,p)\hookrightarrow {\V}\hookrightarrow \ss.$ The operator $\mc{F}(\cdot)$ is locally  Lipschitz continuous; that is, for every bounded set   $D$ in $\V$, there is a positive number $L_{\F}$ such that
\begin{equation}
\normm{\F(\xb_2)-\F(\xb_1)}_{\ss} \le L_{\F} \norm{\xb_2-\xb_1}{\V}, \; \forall\xb_1, \xb_2\in D.
\label{eq-nonlinear-assn}
\end{equation}
When there is no ambiguity, the norm on $\ss$ will not be explicitly indicated. 

For each $\rb\in \as$, the input operator $\mc{B}(\rb)$ is a linear bounded operator that maps the input space $\cs$ into the state space $\ss$ and 
it is continuous with respect to $\rb :$
\begin{equation}\label{B-cont}
\lim_{\rb_n \to \rb_0} \| \mc B (\rb_n ) - \mc B(\rb_0) \| = 0 ,
\end{equation}
where the convergence $\rb_n\to \rb_0$ is with respect to the topology on $\as .$
%it is weakly continuous with respect to $r$ for each $u \in \cs  ;$ that is, for each $u \in \cs, $
%$$\lim_{r_n \to r_0 } \phi  (  B(r_n ) u - B(r_0) u ) = 0 , \quad \forall \phi \in \ss^\prime .$$
\begin{comment}
{For every $\ub\in L^p(0,\tau;\cs)$, the operator $\mc{B}(\cdot)\ub:K\to L^p(0,\tau;\ss)$  is  weakly continuous with respect to $\rb .$ In other words, if $ \lim_{n\to \infty} \rb_n = \rb_0$ in the topology on $\as ,$ then, letting $q$ be such that $\frac{1}{q}+\frac{1}{p} =1 ,$  for all $\phi (\cdot ) \in L^q (0,\tau; \ss ) , $ 
\begin{equation}\label{as-B(r)}
%\text{if } \rb_2\rightharpoonup \rb_1 \text{ then }
\lim_{n \to \infty} \int_0^\tau  \inn{  \mc{B}(\rb_n)\ub (s )  - \mc{B}(\rb_0)\ub (s)  }{\phi (s)}   ds = 0 .
\end{equation}
If $\cs$ is finite-dimensional, as is usual in applications, then this assumption is equivalent to ???? }
\end{comment}

For any positive numbers $R_1$ and $R_2$, define the sets
\begin{flalign}\label{ad sets}
\Bu&=\left\{\ub\in L^p(0,\tau;\cs): \norm{\ub}{p}\le R_1   \right\},\\
B_{\V}(R_2)&=\left\{\xb_0 \in \V: \norm{\xb_0}{\V}\le R_2   \right\}.
\end{flalign}

\begin{definition}\cite[Definition 3.1.i]{bensoussan2015book}(strict solution) The function $\xb(\cdot)$ is said to be a {\em strict solution} of (\ref{eq-IVP}) if $\xb(0)=\xb_0$, $\xb\in \W(0,\tau)$, and $\xb(t)$ satisfies (\ref{eq-IVP}) for almost every $t\in[0,\tau]$.
\end{definition}

\begin{lemma}\cite[Proposition 2.2 and Corollary 2.3]{clement1993abstract}\label{lem-clement}
Let $\tau_0>\tau$ and $p\in (1,\infty)$ be given. If $\A$ has maximal $L^p$ regularity, then there exists a constant $c_{\tau_0}$ independent of $\tau$ such that
for all $\tau\in (0,\tau_0]$ and $\vb\in W^{1,p}(0,\tau;\ss)\cap L^p(0,\tau;D(\A)) ,$
\begin{flalign*}
\norm{\dot{\vb}}{L^2(0,\tau;\ss)}&+\norm{\A\vb}{L^2(0,\tau;\ss)}\\
&\le M_{\tau_0}\left(\norm{\dot{\vb}+\A\vb}{L^2(0,\tau;\ss)}+\norm{\vb(0)}{D_{\A}(1/p,p)}\right) \, .
\end{flalign*}
Furthermore, if $\vb(0)=0$, 
\begin{equation*}
\norm{\vb}{C(0,\tau;D_{\A}(1/p,p))}\le M_{\tau_0}\left(\norm{\dot{\vb}}{L^2(0,\tau;\ss)}+\norm{\A\vb}{L^2(0,\tau;\ss)}\right).
\end{equation*}
\end{lemma}
 
\begin{theorem}\label{thm-existence}
%Let there be $1< p < \infty$ 
For every pair $R_1>0$ , $R_2 >0 ,$ there is $\tau>0$ and $\delta>0$ such that the IVP (\ref{eq-IVP}) admits a unique strict solution $\xb\in\W(0,\tau)$, $\norm{\xb}{\W(0,\tau)}\le \delta$ for all $(\ub,\rb,\xb_0)\in \Bu\times K_{ad}\times B_{\V}(R_2)$.
\end{theorem}
\begin{proof}
The proof of this theorem follows the same line as that of \cite[Theorem 2.1]{clement1993abstract}\label{thm-strict} with some modifications.
Let $\wb$ solve the linear equation
\begin{equation}
\begin{cases}
\dot{\wb}(t)=\A \wb(t)+\F(\xb_0)+\B(\rb)\ub(t), \quad t\in(0,\tau],\\
\wb(0)=\xb_0.
\end{cases}
\end{equation}

Define for an arbitrary  number $\rho>0$ the set 
\begin{equation}
\Sigma_{\rho,\tau}=\left\{\vb\in \W(0,\tau): \vb(0) = \xb_0, \; \norm{\vb-\wb}{\W(0,\tau)}\le \rho \right\}.
\label{eq-sigma-rho}
\end{equation}

Because $\wb (\cdot ) \in \W (0, \tau ) , $ $\wb (\cdot ) \in C ( 0, \tau ; \V ) . $
Define  $\phi(\tau;R_1,R_2)=\norm{\wb -\xb_0}{C(0,\tau;\V)}$ where here $\xb_0$ indicates the constant function in $C(0,\tau ;\V) $ that equals $\xb_0 .$ Note that 
\begin{equation}
\lim_{\tau \to 0 } \phi(\tau;R_1,R_2) =  0 .
\end{equation}

%\begin{equation}
%\delta_{R}(\rho,\tau)=M\rho + \phi(\tau)
%\end{equation}
According to \Cref{lem-clement}, there is a constant $M$ independent of $\tau$ such that
\begin{equation}
\norm{\vb-\xb_0}{C(0,\tau;\V)}\le M\rho+\phi(\tau;R_1,R_2),\quad  \forall \vb\in \Sigma_{\rho,\tau}.
\end{equation}
Consider the mapping $\gamma:\W(0,\tau) \to \W(0,\tau)$, $\xb (\cdot ) \mapsto \vb (\cdot ) $ defined by 
\begin{equation}
\begin{cases}
\dot{\vb}(t)=\A \vb(t)+\F(\xb(t))+\B(\rb)\ub(t), \; t\in(0,\tau],\\
\vb(0)=\xb_0.
\end{cases}
\end{equation}
It will now be shown that for some numbers $\rho$ and $\tau$ the mapping $\gamma$ defines a contraction on $\Sigma_{\rho,\tau}$ and hence has a unique fixed point.

Consider the linear equation
\begin{equation*}
\begin{cases}
\dot{\vb}(t)-\dot{\wb}(t)=\A (\vb(t) -\wb(t) ) +\F(\xb(t)), \; t\in (0,\tau],\\
(\vb-\wb)(0)=0,
\end{cases}
\end{equation*}	
Use \Cref{lem-clement} together with Lipschitz continuity of $\F$, let $L_\F$ be the Lipschitz constant of $\F$ over the ball $B(\xb_0,M\rho+\phi(\tau;R_1,R_2))$. It follows that
\begin{flalign}
\norm{\vb-\wb}{\W(0,\tau)}&\le M\norm{\F(\xb(t))-\F(\xb_0)}{p}\notag \\
&\le M L_{\F} \tau^{\frac{1}{p}} \norm{\xb-\xb_0}{C(0,\tau;\V)}\notag\\
&\le M^2 L_{\F} \tau^{\frac{1}{p}} (M\rho+\phi(\tau;R_1,R_2)) \label{ineq1}.
\end{flalign}
Furthermore,  for any $\xb_1  , \xb_2 \in \Sigma_{\rho,\tau}$,  define $\vb_1=\gamma(\xb_1)$ and $\vb_2=\gamma(\xb_2)$, then \Cref{lem-clement} yields
\begin{flalign}
\norm{\vb_2-\vb_1}{\W(0,\tau)}&\le M \norm{\F(\xb_2)-\F(\xb_1)}{p}\notag\\
&\le M L_{\F} \tau^{\frac{1}{p}} \norm{\xb_2-\xb_1}{C(0,\tau;\V)}\notag \\
&\le M^2 L_{\F} \tau^{\frac{1}{p}} \norm{\xb_2-\xb_1}{\W(0,\tau)}\label{ineq2}.
\end{flalign}
Choose $\rho$ and $\tau$ so that
\begin{gather*}
M^2 L_{\F} \tau^{\frac{1}{p}}< 1,\\
M^2 L_{\F} \tau^{\frac{1}{p}} (M\rho+\phi(\tau;R_1,R_2))\le \rho.
\end{gather*}
The Contraction Mapping Theorem ensures that the mapping $\gamma$ has a unique fixed point in $\Sigma_{\rho,\tau}$. This fixed point is the unique solution $\xb$ to (\ref{eq-IVP}). Also, 
from the definition (\ref{eq-sigma-rho}), every $\xb$ in $\Sigma_{\rho,\tau}$ satisfies  
\begin{equation}
\norm{\xb}{\W(0,\tau)}\le \norm{\wb}{\W(0,\tau)}+\rho.
\end{equation}
Let $L_\F$ be the Lipschitz constant of $\F$ over the ball $B(0, \| \xb_0 \| )$. Proposition 2.2 in \cite{clement1993abstract} yields
\begin{flalign*}
&\norm{\wb}{\W(0,\tau)}\le M(\norm{\xb_0}{\V}+\norm{\F(\xb_0)+\B(\rb)\ub(t)}{p})\notag \\
&\quad \le M(\norm{\xb_0}{\V}+\tau^{\frac{1}{p}}L_{\F}\norm{\xb_0}{\V}+\norm{\B(\rb)}{\mc L (\ss,\cs)}\norm{\ub(t)}{p})\notag\\
&\quad \le \underbrace{M(R_2+\tau^{\frac{1}{p}}L_{\F}R_2+R_1\max_{\rb\in K_{ad}}\norm{\B(\rb)}{\mc L (\ss,\cs)}) }_{\delta} .  \notag
% &\quad < \infty.     \notag
\end{flalign*}
Defining 
$$\delta =  M(R_2+\tau^{\frac{1}{p}}L_{\F}R_2+R_1\max_{\rb\in K_{ad}}\norm{\B(\rb)}{\mc L (\ss,\cs)})  ,$$ yields the required 
 upper-bound  on $\norm{\xb}{\W(0,\tau)}$.
\end{proof}
%\begin{assumption}\label{as-A}
%There exists a bounded, closed, convex sets $B_{\V}(R_2)\subset \V$ with non-empty interior such that the IVP (\ref{eq-IVP}) has a unique strict solution over $[0,\tau]$ for all $\xb_0\in B_{\V}(R_2)$, $\ub\in \Bu$, and $\rb\in K_{ad}$. 
%\end{assumption}
 
\begin{definition}
Let $\xb(t)$ be the strict solution to (\ref{eq-IVP}). The mapping $\mc S(\ub,\rb,\xb_0):\Bu\times K_{ad} \times B_{\V}(R_2) \to\W(0,\tau)$, $(\ub(t),\rb,\xb_0)\mapsto \xb(t)$, is called the solution map. 
\end{definition}

An embedding $D(\A)\hookrightarrow\ss$ where $D(\A)$ is compact in $\ss$ ensures that the space $W^{1,p}(0,\tau;\ss)\cap L^p(0,\tau,D(\A))$ is compactly embedded in $c^{s}(0,\tau;{\V})$, $0\le s<1$
 \cite[Theorem 5.2]{amann2000compact}.  
 Since $c^{s}(0,\tau;{\V})\hookrightarrow C(0,\tau;\V)$, it follows that the space $W^{1,p}(0,\tau;\ss)\cap L^p(0,\tau,D(\A))$ is compactly embedded in $C(0,\tau;{\V})$.

\begin{theorem}\label{thm-weak}
If the embedding $D(\A)\hookrightarrow\ss$ is compact then the solution map is weakly continuous with respect to $(\ub(t),\rb,\xb_0)\in L^p(0,\tau;\cs)\times \as \times \V$.
\end{theorem}
\begin{proof}
The weak continuity of the solution map with respect to $\ub(t)$ is shown in \cite[Lemma 2.12]{meyer2017optimal}. Weak continuity with respect to $(\ub(t),\rb,\xb_0)$ follows from a similar proof. 
 Choose any weakly convergent sequences $\{\ub_n(t)\}\subset L^p(0,\tau;\cs)$, $\{ \xb_0^n\} \subset \V $, and $\{\rb_n \}\subset \as .$   Since sets $\Bu$ and $B_{\V}(R_2)$ are bounded, closed, convex subsets of Banach spaces $L^p(0,\tau;\cs)$ and $\V$, respectively; these sets are weakly closed \cite[Theorem 2.11]{troltzsch2010optimal}. This implies that there are $\ub^o\in \Bu$ and $\xb_0\in B_{\V}(R_2)$ such that
\begin{flalign}\label{eq-3}
\ub_n & \rightharpoonup \ub^o \text{ in } \Bu,\\
\xb_0^n &\rightharpoonup\xb_0 \text{ in } B_{\V}(R_2).
\end{flalign}
Since the set $K_{ad}$ is a compact subset of $\as$
\begin{equation}
\rb_n\to \rb^o \text{ in } K_{ad}.
\end{equation}
It will be shown that $\B_{\rb_n}\ub_n(t)$ converges weakly to $\B_{\rb^o}\ub^o(t)$ in $L^p(0,\tau;\ss)$. For every $\zb \in L^q(0,\tau;\ss)$, $1/q=1-1/p$,
\begin{flalign}\notag
I:=&\inn{\zb}{\B_{\rb_n}\ub_n-\B_{\rb^o}\ub^o}_{L^q(0,\tau;\ss^{^*}),L^p(0,\tau;\ss)}\\
=&\inn{\zb}{\B_{\rb_n}\ub_n-\B_{\rb^o}\ub_n}_{L^q(0,\tau;\ss^{^*}),L^p(0,\tau;\ss)}\\
&+\inn{\zb}{\B_{\rb^o}\ub_n-\B_{\rb^o}\ub^o}_{L^q(0,\tau;\ss^{^*}),L^p(0,\tau;\ss)}.\notag
\end{flalign}
Taking the adjoint and norm yield
\begin{flalign*}
I \le \norm{\B_{\rb_n}-\B_{\rb^o}}{\mc L(\cs,\ss)}\int_0^\tau \norm{\ub_n(t)}{\cs}\normm{\zb(t)}dt\\
+{\Big |}\int_0^\tau\inn{\B^*_{\rb^o}\zb(t)}{\ub_n(t)-\ub^o(t)}_{\cs^{^*},\cs}dt{\Big |}.
\end{flalign*}
Use H\"older inequality and let $\vb(t)=\B^*_{\rb^o}\zb(t)$, it follows that
\begin{flalign*}
I \le \norm{\B_{\rb_n}-\B_{\rb^o}}{\mc L(\cs,\ss)}\norm{\ub_n}{L^p(0,\tau;\cs)}\norm{\zb}{L^q(0,\tau;\ss)}\\
+|\inn{\ub_n-\ub^o}{\vb}_{L^p(0,\tau;\cs),L^q(0,\tau;\cs^{^*})}|.
\end{flalign*}
The convergence of the first term follows from \eqref{B-cont}. The second term converges to zero because $\ub_n\rightharpoonup \ub^o$ in $L^p(0,\tau;\cs)$. Combining these yields
\begin{equation}\label{weak-B}
\B_{\rb_n}\ub_n\rightharpoonup \B_{\rb^o}\ub^o \text{ in }L^p(0,\tau;\ss).
\end{equation}

Using \Cref{thm-existence}, the corresponding solution $\xb_n(t)$ is a bounded sequence in the reflexive Banach space $L^p(0,\tau;D(\A))\cap W^{1,p}(0,\tau;\ss)$. Thus, there is a subsequence of $\xb_n(t)$ such that 
\begin{equation}\label{eq-1}
\xb_{n_k}\rightharpoonup\xb \text{ in }\W(0,\tau).
\end{equation}
This in turn implies that the sequence $\xb_n(t)$ strongly converges to $\xb(t)$ in $C(0,\tau;{\V})$. This together with Lipschitz continuity of $\F(\cdot)$ yields 
\begin{equation}\label{eq-d1}
\F(\xb_{n_k}(t)) \to \F(\xb(t)) \quad \text{in } L^p(0,\tau;\ss).
\end{equation}
This strong convergence also yields weak convergence in the same space, that is
\begin{equation}\label{eq-d2}
\F(\xb_{n_k}(t)) \rightharpoonup \F(\xb(t)) \quad \text{in } L^p(0,\tau;\ss).
\end{equation}
Now apply (\ref{eq-3}), \eqref{weak-B}, (\ref{eq-1}), and (\ref{eq-d2}) to the IVP (\ref{eq-IVP}); take the limit; notice that a solution to the IVP is unique; it follows that $\xb=\mc S(\ub,\rb,\xb_0)$. Deleting elements $\{\xb_{n_k}(t)\}$ from $\{\xb_n(t)\}$ and repeating the previous processing, knowing that a weak limit is unique, it follow that $\xb_n(t)\rightharpoonup \xb(t)$ in $\W(0,\tau)$.
\end{proof}

\section{Optimal Actuator Design}
Consider a  cost function $J(\xb,\ub,\rb):\W(0,\tau)\times L^p(0,\tau;\cs )\times \as \to \mathbb{R}$ 
 that is bounded below and weakly lower-semicontinuous with respect to $\xb$, $\ub$, and $\rb$.
For a fixed initial condition $\xb_0\in B_{\V}(R_2)$, consider the following optimization problem over the admissible input set $U_{ad}$ and actuator design set $K_{ad}$
\begin{equation}
\left\{ \begin{array}{ll}
\min&J(\xb,\ub,\rb)\\
\text{s.t.}& \xb=\mc S(\ub,\rb,\xb_0),\\
&(\ub,\rb) \in U_{ad}\times K_{ad}.
\end{array} \right. \tag{P} \label{eq-optimal problem}
\end{equation}%%
The set $U_{ad}$ will be assumed a convex and closed set contained in the interior of $\Bu$.
\begin{theorem} \label{thm-existence optimizer}
For every $\xb_0\in B_{\V}(R_2)$, there exists a control input $\ub^o\in U_{ad}$ together with an actuator design $\rb^o\in K_{ad}$ that solve the optimization problem \eqref{eq-optimal problem}.
\end{theorem}
\begin{proof}
The proof of this theorem follows from  standard analysis; see for example, \cite[Theorem 1.45]{hinze2008optimization} and \cite[Theorem 4.1]{edalatzadehSICON} for a similar argument. 
 Define 
\begin{equation}
j(\xb_0):=\inf_{(\ub,\rb) \in U_{ad}\times K_{ad}}J(\mc S(\ub,\rb,\xb_0),\ub,\rb).
\end{equation}
and let $(\ub_n,\rb_n)$ be the minimizing sequence:
\begin{equation}
\lim_{n\to \infty}J(\mc S(\ub_n,\rb_n,\xb_0),\ub_n,\rb_n)=j(\xb_0).
\end{equation}
The set $U_{ad}$ is closed and convex in the reflexive Banach space $L^p(0,\tau;\cs)$, so it is weakly closed. This implies that there is a subsequence of $\ub_n$, denote it by the same symbol, that converges weakly to some elements $\ub^o$ in $U_{ad}$. Because of compactness of $K_{ad}$, there is also a subsequence of $\rb_n$, denote it by the same symbol, that strongly converges to $\rb^o$.
\Cref{thm-existence} and \Cref{thm-weak} state that the solution map is bounded and weakly continuous in each variable. Thus, the corresponding state $\xb_n=\mc S(\ub_n,\rb_n;\xb_0)$ also weakly converges to $\xb^o= \mc S(\ub^o,\rb^o;\xb_0)$ in $\W(0,\tau)$. The cost function is weakly lower semi-continuous with respect to each $\xb$, $\ub$, and $\rb$, this ensures that $(\xb^o,\ub^o,\rb^o)$ minimizes the cost function. Therefore, $(\ub^o,\rb^o)$ is a solution to the optimization problem \eqref{eq-optimal problem}.
\end{proof}

\begin{definition}\cite[Definition 1.29]{hinze2008optimization}
The operator $\G:\ss\to \Yb$ is said to be G\^ateaux differentiable at $\xb\in \ss$ in the direction $\pb\in \ss$, if there is a linear bounded operator $\G'_\xb$ such that for all real $\epsilon$
\begin{equation}
\lim_{\epsilon\to 0}{\|\G(\xb+\epsilon\pb)-\G(\xb)-\epsilon\G'_\xb\pb\|_{\Yb}}=0.
\end{equation}
\end{definition}

The optimality conditions are derived next after assuming that the problem has certain properties. Consider the assumptions:

\begin{enumerate}
\item[A1.] The spaces $\ss$ and $\cs$ are Hilbert spaces and $p=2$. The space $\as $ is a Banach space.
\item[A2.] Let $a:\V\times \V\to \mathbb{C}$ be a sesquilinear form (see \cite[Chapter 4]{lang2012real}), where $\V\hookrightarrow \ss$, and let there be positive numbers $\alpha$ and $\beta$ such that
\begin{flalign*}
|a(\xb_1,\xb_2)|&\le \alpha \norm{\xb_1}{\V}\norm{\xb_2}{\V}, &\forall&\xb_1,\xb_2\in \V,\\
\text{Re} \; a(\xb,\xb)&\ge \beta \norm{\xb}{\V}^2,  &\forall&\xb\in \V.
\end{flalign*}
The operator $\A$ has an extension to $\bar{\A}\in\mc L(\V,\V^{^*})$ described by
\begin{equation}
\inn{\bar{\A} \vb}{\wb}_{\V^{^*},\V}=a(\vb,\wb), \quad \forall \vb, \wb \in \V,
\end{equation}
where $\V^{^*}$ denotes the dual of $\V$ with respect to pivot space $\ss$.
\item[A3.] \label{as-diff J}The cost function $J(\xb,\ub,\rb)$ is continuously Fr\'echet differentiable with respect to each variable.
\item[A4.] \label{as-diff F} The nonlinear operator $\F(\cdot)$ is G\^ateaux differentiable. Indicate the G\^ateaux derivative of $\F(\cdot)$ at $\xb$ in the direction $\pb$ by ${\F}_{\xb}^\prime\pb$. Furthermore, the mapping $\xb\mapsto {\F}_{\xb}^\prime$ is bounded; that is, bounded sets in ${\V}$ are mapped to bounded sets in $\mc L({\V},\ss)$.
\item[A5.] \label{as:diff B} The control operator $\mc{B}(\rb)$ is G\^ateaux differentiable with respect to $\rb$ from $K_{ad}$ to $\mc{L}(\cs,\ss)$. Indicate the G\^ateaux derivative of $\mc{B}(\rb)$ at $\rb^o$ in the direction $\rb$ by $\mc{B}'_{\rb^o}\rb$. Furthermore, the mapping $\rb^o\mapsto \B'_{\rb^o}$ is bounded; that is, bounded sets in $\as$ are mapped to bounded sets in $\mc L(\as,\mc L(\cs,\ss))$.
\end{enumerate}

Using these assumptions, the G\^ateaux derivative of the solution map with respect to a trajectory $\xb(t)=\mc S(\ub(t),\rb,\xb_0)$ is calculated. The resulting map is a time-varying linear IVP. Let $\bm g\in L^p(0,\tau;\ss)$, consider the time-varying system
\begin{equation}\label{eq-time var}
\begin{cases}
\dot{{\bm{h}}}(t)=(\A+{\F}_{\xb(t)}^\prime) {\bm{h}}(t)+\bm g(t),\\
{\bm{h}}(0)=0.
\end{cases}
\end{equation}

\begin{lemma}\cite[Corollary 5.2]{dier2015}\label{lem-estimate}
Let assumptions A1 and A2 hold. For any $\tau>0$, let $\mc P(\cdot):[0,\tau]\to \mc L (\V,\ss)$ be such that $\mc P(\cdot)\xb$ is weakly measurable for all $\xb\in \V$, and there exists an integrable function $h:[0,\tau]\to [0,\infty)$ such that $\norm{\mc P(t)}{\mc L(\V,\ss)}\le h(t)$ for all $t\in [0,\tau]$. Then for every $\xb_0\in \V$ and $\bm g\in L^2(0,\tau;\ss)$, there exists a unique $\xb$ in $\W(0,\tau)$ such that
\begin{equation}
\begin{cases}
\dot{\xb}(t)=(\A  +\mc P(t)) \xb(t)+\bm g(t),\\
\xb(0)=\xb_0.
\end{cases}
\end{equation}
Moreover, there exists a constant $c>0$ independent of $\xb_0$ and $\bm g(t)$ such that
\begin{equation}
\norm{\xb}{\W(0,\tau)}^2\le c \left(\norm{\bm g}{L^2(0,\tau;\ss)}^2+\norm{\xb_0}{\V}^2\right).
\end{equation} 
\end{lemma}

Since $\W(0,\tau)$ is embedded in $C(0,\tau;\V)$, the state $\xb(t)$ is bounded in $\V$ for all $t\in [0,\tau]$. This together with G\^ateaux differentiablity of $\F(\cdot)$ ensures that there is a positive number $M_{\F}$ such that
\begin{equation}
\sup_{t\in [0,\tau]}\norm{{\F}_{\xb(t)}^\prime}{\mc L(\V,\ss)}\le M_{\F}.
\end{equation} 
Thus, replacing the operator $\mc P(t)$ with ${\F}_{\xb(t)}^\prime$ and noting that
\begin{equation}\label{d1}
\norm{\mc P(t)}{\mc L(\V,\ss)}\le M_{\F},
\end{equation}
shows that the conditions of \Cref{lem-estimate} hold. Thus, there is a positive number $c$ independent of $\bm g$ such that
\begin{equation}\label{eq-estimate}
\|{\bm{h}}\|_{\W(0,\tau)}\le c \norm{\bm g}{L^2(0,\tau;\ss)}.
\end{equation}

\begin{proposition}\label{prop-diff}
Under assumptions A1-A5, the solution map $\mc S(\ub(t),\rb;\xb_0)$ is G\^ateaux differentiable with respect to each $\ub(t)$ and $\rb$ in $U_{ad}\times K_{ad}$. Let $\xb(t)=\mc S(\ub(t),\rb,\xb_0)$.
\begin{enumerate}
\item[a.] The G\^ateaux derivative of $\mc S(\ub(t),\rb;\xb_0)$ at $\rb$ in the direction $\tilde{\rb}$ is the mapping $\mc S'_\rb :\as\to L^2(0,\tau;D(\A))\cap W^{1,2}(0,\tau;\ss)$, $\tilde{\rb}\mapsto \zb(t)$, where $\zb(t)$ is the strict solution to
\begin{equation}
\begin{cases}
\dot{\zb}(t)=(\A+{\F}_{\xb(t)}^\prime) \zb(t)+(\B'_\rb\tilde{\rb})\ub(t),\\
\zb(0)=0.
\end{cases}
\end{equation}
\item[b.] The G\^ateaux derivative of $\mc S(\ub(t),\rb;\xb_0)$ at $\ub(t)$ in the direction $\tilde{\ub}(t)$ is the mapping $\mc S'_\ub :L^2(0,\tau;\cs )\to L^2(0,\tau;D(\A))\cap W^{1,2}(0,\tau;\ss)$, $\tilde{\ub}(t)\mapsto {\bm{h}}(t)$, where ${\bm{h}}(t)$ is the strict solution to
\begin{equation}
\begin{cases}
\dot{{\bm{h}}}(t)=(\A+{\F}_{\xb(t)}^\prime) {\bm{h}}(t)+\B(\rb)\tilde{\ub}(t),\\
{\bm{h}}(0)=0.
\end{cases}
\end{equation}
\end{enumerate}
\end{proposition}
\begin{proof} 
a) Let $\epsilon$ be sufficiently small such that $\rb+\epsilon\tilde{\rb}\in K_{ad}$. Define $\xb_\epsilon(t)=\mc S(\ub(t),\rb+\epsilon\tilde{\rb},\xb_0)$, this state solves
\begin{equation}
\begin{cases}
\dot{\xb}_\epsilon(t)=\mc{A}\xb_\epsilon(t)+\mc{F}(\xb_\epsilon(t))+\mc{B}(\rb+\epsilon\tilde{\rb})\ub(t),\quad t>0,\\
 \xb_\epsilon(0)=\xb_0.
\end{cases}
\label{eq-xe}
\end{equation}
Similarly, $\xb (t)=  \mc S(\ub(t),\rb,\xb_0)$  solves \eqref{eq-xe} with $\epsilon =0 .$ Define $\bm{\be}_{\F}(t)$ and $\bm{\be}_{\B}$ as 
\begin{subequations}
\begin{flalign}
\bm{\be}_{\F}(t)&\coloneqq \frac{1}{\epsilon}\left(\F(\xb(t))-\F(\xb_\epsilon(t))-{\F}_{\xb(t)}^\prime(\xb(t)-\xb_\epsilon(t))\right),\label{eF}\\
\bm{\be}_{\B} &\coloneqq \frac{1}{\epsilon}\left(\B(\rb+\epsilon\tilde{\rb})-\B(\rb)\right)-\B'_{\rb}\tilde{\rb}.
\end{flalign}
\end{subequations}
The state $\be(t)=(\xb(t)-\xb_\epsilon(t))/\epsilon-\zb(t)$ satisfies
\begin{equation}\label{edot}
\begin{cases}
\dot{\be}(t)=(\mc{A}+{\F}_{\xb(t)}^\prime)\be(t)+\bm{\be}_{\F}(t)+\bm{\be}_{\B}\ub(t),\; t>0,\\
\be(0)=0.
\end{cases}
\end{equation}
Assumption A4 and A5 ensure that as $\epsilon\to 0$
\begin{subequations}
\begin{gather}
\normm{\bm{\be}_{\F}(t)}\to 0, \quad \forall t\in [0,\tau],\label{e(t)2}\\
\norm{\bm{\be}_{\B}}{\mc{L}(\cs,\ss)} \to 0.\label{eB}
\end{gather}
\end{subequations}%

It will be shown that $\lim_{\epsilon\to 0} \norm{\be}{\W(0,\tau)}=0.$  First, consider $\xb(t)-\xb_\epsilon(t)$, which satisfies
\begin{flalign}\label{eq-7}
\begin{cases}
\dot{\xb}(t)-\dot{\xb}_\epsilon(t)=&\A(\xb(t)-\xb_\epsilon(t))+\F(\xb(t))-\F(\xb_\epsilon(t))\notag \\
&+\left(\B(\rb)-\B(\rb+\epsilon\tilde{\rb})\right)\ub(t),\\
\xb(0)-\xb_\epsilon(0)=&0.\notag
\end{cases}
\end{flalign}
Lemma \ref{lem-clement} implies that there is a number $c_\tau$ depending only on $\tau$ such that for all $t\in [0,\tau]$
\begin{flalign}
&\norm{\xb(t)-\xb_\epsilon(t)}{\V} \\
&\quad \le c_\tau \left(\norm{\dot{\xb}-\dot{\xb}_\epsilon}{L^2(0,t;\ss)}+\norm{\A(\xb-\xb_\epsilon)}{L^2(0,t;\ss)}\right).\label{eq11}
\end{flalign} 
Also, use \cite[Proposition 2.2]{clement1993abstract}, there is a number $d_\tau$ depending only on $\tau$ such that for all $t\in [0,\tau]$
\begin{flalign}\label{eq12}
\norm{\dot{\xb}-\dot{\xb}_\epsilon}{L^2(0,t;\ss)}&+\norm{\A(\xb-\xb_\epsilon)}{L^2(0,t;\ss)} \\
&\quad \le d_\tau \norm{\dot{\xb}-\dot{\xb}_\epsilon-\A(\xb-\xb_\epsilon)}{L^2(0,t;\ss)}\notag.
\end{flalign}
Combine (\ref{eq11}) and (\ref{eq12}) to obtain
\begin{flalign}
\norm{\xb(t)-\xb_\epsilon(t)}{\V}&\le c_\tau d_\tau \norm{\F(\xb)-\F(\xb_\epsilon)}{L^2(0,t;\ss)} \\
&\quad +c_\tau d_\tau \norm{\left(\B(\rb)-\B(\rb+\epsilon\tilde{\rb})\right)\ub}{L^2(0,t;\ss)}.\notag
\end{flalign}
\Cref{thm-existence} implies that the states $\xb(t)$ and $\xb_\epsilon(t)$ belong to some bounded set in $\W(0,\tau)$ and so in $C(0,\tau,\V)$. Let $D\subset \V$ be a bounded set that contains the trajectories $\xb(t)$ and $\xb_\epsilon(t)$.  Let $L_\F$ be the Lipschitz constant of $\F(\cdot)$ on $D$. 
Since the set $K_{ad}$ is compact and $\B(\rb)$ satisfies assumption A5, the number $L_{\B}$ defined as
\begin{equation}
L_{\B}=\sup_{\rb\in K_{ad}}\norm{\B'_{\rb}}{\mc L(\as,\mc L(\cs,\ss))}.
\end{equation}
is finite. This together with \cite[Theorem 12.1.1 and Corollary 3]{wouk1979course} yields
\begin{equation}
\norm{\B(\rb)-\B(\rb_\epsilon)}{\mc L(\cs,\ss)}\le L_{\B}\norm{\rb-\rb_\epsilon}{\as}\le L_{\B}\epsilon.
\end{equation}
Use these to obtain the inequality
\begin{flalign}\label{eq-8}
\norm{\xb(t)-\xb_\epsilon(t)}{\V}^2&\le 2c^2_\tau d^2_\tau L^2_\F \int_0^t \norm{\xb(s)-\xb_\epsilon(s)}{\V}^2ds \\
&\quad +2c^2_\tau d^2_\tau L^2_{\B}\epsilon^2 \norm{\ub}{L^2(0,\tau;\cs)}^2\notag
\end{flalign}  
Applying Gronwall's lemma yields 
\begin{equation}\label{eq-9}
\norm{\xb(t)-\xb_\epsilon(t)}{\V}\le \sqrt{2}e^{c^2_\tau d^2_\tau L^2_\F}c_\tau d_\tau L_{\B}\epsilon\norm{\ub}{L^2(0,\tau;\cs)}.
\end{equation}

Define $$M_{\F}:=\sup_{t\in [0,\tau]} \|{\F}_{\xb(t)}^\prime\|_{\mc L(\V,\ss)}.$$ 
Assumption A4 ensures that $M_{\F}$ is finite. Take the norm of the right side of \eqref{eF} in $\ss$. It follows that
\begin{equation}\label{eq13}
\normm{\be_{\F}(t)} \le (L_{\F}+M_{\F}c_e)\sqrt{2}e^{c^2_\tau d^2_\tau L^2_\F}c_\tau d_\tau L_{\B}\norm{\ub}{L^2(0,\tau;\cs)}.
\end{equation}
This and (\ref{e(t)2}) together with the Bounded Convergence Theorem ensure that
\begin{equation}\label{eq-10}
\lim_{\epsilon\to 0}\int_0^\tau \normm{\be_{\F}(t)}^2dt= 0.
\end{equation}
Statements (\ref{eq-10}) and (\ref{eB}), and \Cref{lem-estimate} can be applied to conclude
\begin{equation}
\lim_{\epsilon\to 0} \norm{\be}{\W(0,\tau)}=0.
\end{equation}
This shows that $\mc S(\ub,\rb,\xb_0)$ is G\^ateaux differentiable at $\rb$ in the direction $\tilde{\rb}$ with derivative $\zb(t)=\mc S'_\rb \tilde{\rb}$.

b) This part is proven in \cite[Theorem 3.4]{meyer2017optimal} assuming that $\partial_t+\A$ is invertible. However, the result is still true without assuming the invertibility of $\partial_t+\A$. Let $\epsilon$ be sufficiently small such that $\ub+\epsilon\tilde{\ub}\in U_{ad}$. Define $\xb_\epsilon(t)=\mc S(\ub(t)+\epsilon\tilde{\ub}(t),\rb,\xb_0)$, this state solves
\begin{equation}
\begin{cases}
\dot{\xb}_\epsilon(t)=\mc{A}\xb_\epsilon(t)+\mc{F}(\xb_\epsilon(t))+\mc{B}(\rb)(\ub(t)+\epsilon\tilde{\ub}(t)),\quad t>0,\\
 \xb_\epsilon(0)=\xb_0.
\end{cases}
\label{eq-xe}
\end{equation}
Let $\be(t)=(\xb(t)-\xb_\epsilon(t))/\epsilon-\bm{h}(t)$. Following the same steps as in part (a) yields
\begin{equation}
\lim_{\epsilon\to 0} \norm{\be}{\W(0,\tau)}=0.
\end{equation}
This means that $\mc S(\ub,\rb,\xb_0)$ is G\^ateaux differentiable at $\ub$ in the direction $\tilde{\ub}$ with derivative ${\bm{h}}(t)=(\mc S'_\ub \tilde{\ub})(t)$.
\end{proof}

Assumption A1 implies that the  dual of each of $\ss$ and $\cs$ will be identified with the space itself. For each $\ub$, the operator $(\mc{B}'_{\rb^o}\ub)^*:\ss\to \mathbb{K}^*$ is defined by
\begin{equation}\label{B-adj}
\inn{(\mc{B}'_{\rb^o}\ub)^*\pb}{\rb}_{\as^*,\as}=\inn{\pb}{(\mc{B}'_{\rb^o}\rb)\ub}, \; \forall (\ub,\pb,\rb)\in  \cs\times\ss\times \mathbb{K}.\notag
\end{equation}

\begin{theorem}\label{thm-optimality}
Suppose assumptions A1-A5 hold, and writing the derivatives $J'_{\xb}$, $J'_{\ub}$, and $J'_{\rb}$ by elements $\bm j_\xb\in \W(0,\tau)^*$, $\bm j_\ub\in L^2(0,\tau;\cs)$ and $\bm j_\rb \in \as^*$, respectively. 
For any initial condition $\xb_0\in\ss$, let the pair $(\ub^o,\rb^o)\in U_{ad}\times K_{ad}$ be a local minimizer of the optimization problem \eqref{eq-optimal problem} with the optimal trajectory $\xb^o=\mc{S}(\ub^o;\rb^o,\xb_0)$ and let 
 $\pb^o(t)$  indicate the strict solution in $\W(0,\tau)^*$ of the final value problem
\begin{equation}\label{adj}
\dot{\pb}^o(t)=-(\mc{A}^*+{\F_{\xb^o(t)}^\prime}^*)\pb^o(t)-\bm j_{\xb^o}(t), \quad \pb^o(\tau)=0.
\end{equation} 
Then $(\ub^o,\rb^o)$ satisfy
\begin{subequations}
\begin{equation}\notag
\begin{aligned}
&\inn{\bm j_{\ub^o}+\mc{B}^*(\rb^o)\pb^o}{\ub-\ub^o}_{L^2(0,\tau;\cs )}\ge 0,\\
&\inn{\bm j_{\rb^o}+\int_0^{\tau} (\mc{B}'_{\rb^o}\ub^o(t))^*\pb^o(t)\, dt}{\rb-\rb^o}_{\as^*,\as}\ge 0.
\end{aligned}
\end{equation}
\end{subequations}
for all $\ub \in U_{ad}$ and $\rb\in K_{ad}$.
\end{theorem}
\begin{proof}
Let $$\mc G(\ub,\rb)=J(\mc S(\ub,\rb,\xb_0),\ub,\rb).$$
The G\^ateaux derivative of $\mc G(\ub,\rb)$ with respect to $\ub$ has been obtained in the proof of \cite[Proposition 4.13]{meyer2017optimal}. Using the chain rule to take the G\^ateaux derivative of $\mc G(\ub,\rb)$ at $\ub^o$ in the direction $\tilde{\ub}$ yields
\begin{equation}
\mc G'_{\ub^o}\tilde{\ub}= J'_{\ub^o}\tilde{\ub} + J'_{\xb^o}\mc S'_{\ub^o}\tilde{\ub} .\label{opt1}
\end{equation}
Identify the functionals $\mc G'_{\ub^o}:L^2(0,\tau;\cs)\to \R$ and $J'_{\ub^o}:L^2(0,\tau;\cs)\to \R$ with elements of $L^2(0,\tau,\cs)$. That is
\begin{flalign}
\mc G'_{\ub^o}\tilde{\ub}&=\inn{\bm g_{\ub^o}}{\tilde{\ub}}_{L^2(0,\tau;\cs)},\label{Ju}\\
J'_{\ub^o}\tilde{\ub}&=\inn{\bm j_{\ub^o}}{\tilde{\ub}}_{L^2(0,\tau;\cs)}.\label{ju}
\end{flalign}
Also, identifying the functional $J'_{\xb^o}:L^2(0,\tau;\ss)\to \R$ with an element of $\W(0,\tau)^*=L^2(0,\tau;D(\A^*))\cap W^{1,2}(0,\tau;\ss)$ yields
\begin{equation}
J'_{\xb^o}\mc S'_{\ub^o}\tilde{\ub}=\inn{\bm j_{\xb^o}}{\mc S'_{\ub^o}\tilde{\ub}}_{L^2(0,\tau;\ss)}.\label{jx}
\end{equation}
The adjoint operator $\mc S'^*_{\ub^o}$ can be obtained as follows. Use \eqref{adj} in the following inner product and let $\bm{h}(t)=\mc S'_{\ub^o}\tilde{\ub}$
\begin{flalign*}
&\inn{\bm j_{\xb^o}}{\mc S'_{\ub^o}\tilde{\ub}}_{L^2(0,\tau;\ss)}\\
&\quad =\int_0^\tau \inn{-\dot{\pb}^o(t)-(\mc{A}^*+{\F_{\xb^o(t)}^\prime}^*)\pb^o(t)}{\bm{h}(t)}dt.
\end{flalign*}
Taking the adjoint and integration by parts yield
\begin{flalign*}
&\inn{\bm j_{\xb^o}}{\mc S'_{\ub^o}\tilde{\ub}}_{L^2(0,\tau;\ss)}\\
&\quad=\int_0^\tau \inn{\pb^o(t)}{\dot{\bm{h}}(t)-(\A+{\F_{\xb^o(t)}^\prime}) \bm{h}(t)}dt\\
&\quad=\int_0^\tau \inn{\pb^o(t)}{\B(\rb)\tilde{\ub}(t)}dt\\
&\quad=\int_0^\tau \inn{\B^*(\rb)\pb^o(t)}{\tilde{\ub}(t)}_{\cs}dt.
\end{flalign*}
This implies 
\begin{equation}
\mc S'^*_{\ub^o} \bm j_{\xb^o}= \B^*(\rb)\pb^o(t).\label{adj3}
\end{equation}
Combine \eqref{Ju}, \eqref{ju}, \eqref{jx} and use \eqref{adj3}, equation \eqref{opt1} is written using the functionals as
\begin{equation}
\inn{\bm g_{\ub}}{\tilde{\ub}}_{L^2(0,\tau;\cs)}=\inn{\bm j_{\ub^o}+\mc{B}^*(\rb^o)\pb^o}{\tilde{\ub}}_{L^2(0,\tau;\cs)}.
\end{equation}
Applying \cite[Theorem 1.46]{hinze2008optimization} and letting $\tilde{\ub}=\ub-\ub^o$ for all $\ub \in U_{ad}$ yields
\begin{equation}
\inn{\bm j_{\ub^o}+\mc{B}^*(\rb^o)\pb^o}{\ub-\ub^o}_{L^2(0,\tau;\cs )}\ge 0.
\end{equation}

Using the  chain rule to take the G\^ateaux derivative of $\mc G(\ub,\rb)$ at $\rb^o$ in the direction $\tilde{\rb}$ yields
\begin{equation}
\mc G'_{\rb^o}\tilde{\rb}= J'_{\rb^o}\tilde{\rb} + J'_{\xb^o}\mc S'_{\rb^o}\tilde{\rb} .
\end{equation}
Write the functionals $\mc G'_{\rb^o}:\as\to \R$ and $J'_{\rb^o}:\as \to \R$ as elements of $\bm g_{\rb^o}$ and $\bm j_{\rb^o}$ in $\as^*$, respectively, and take the adjoint of $\mc S'_{\rb^o}$. It follows that
\begin{equation}
\bm g_{\rb^o}=\mc S'^*_{\rb^o}\bm j_{\xb^o}(t)+\bm j_{\rb^o}.
\end{equation}

%The adjoint operator $\mc S'^*_{\rb^o}$ will be derived explicitly. 
An explicit representation of the adjoint operator $\mc S'^*_{\rb^o}$ will be derived. Consider the inner product 
\begin{equation*}
\inn{\bm j_{\xb^o}}{\mc{S}'_{\rb^o}\tilde{\rb}}_{L^2(0,\tau;\ss )}=\int_0^\tau \inn{\bm j_{\xb^o}(t)}{\mc{S}'_{\rb^o}\tilde{\rb}}dt.
\end{equation*}
Write $\zb (t)=\mc{S}'_{\rb^o}\tilde{\rb}$. Substitute for $\bm j_{\xb^o}(t)$ from \eqref{adj} into this integral. 
Perform integration by parts to obtain
\begin{flalign}
&\int_0^\tau\inn{-\dot{\pb}^o(t)-(\mc{A}^*+{\F_{\xb^o(t)}^\prime}^*)\pb^o(t)}{\zb(t)}dt\notag \\
&\quad =\int_0^\tau \inn{\pb^o(t)}{\dot{\zb}(t)-(\mc{A}+{\F}_{\xb^o(t)}^\prime)\zb(t)}dt \notag \\
&\quad =\int_0^\tau \inn{\pb^o(t)}{(\B'_{\rb^o}\tilde{\rb})\ub^o(t)}dt\notag \\
&\quad =\inn{\int_0^\tau (\B'_{\rb^o}\ub^o(t))^*\pb^o(t)dt}{\tilde{\rb}}_{\as^{*},\as}.
\end{flalign}
Thus, 
\begin{equation}
\mc S'^*_{\rb^o}\bm j_{\xb^o}(t)=\int_0^\tau (\B'_{\rb^o}\ub^o(t))^*\pb^o(t)dt.
\end{equation}
As a result, the G\^ateaux derivative of $\mc G(\ub,\rb)$ at $\rb^o$ in the direction $\tilde{\rb}$ is
\begin{equation}
\bm g'_{\rb^o}=\int_0^\tau (\B'_{\rb^o}\ub^o(t))^*\pb^o(t)dt+\bm j_{\rb^o}.
\end{equation}
The optimality conditions now follow by substituting the G\^ateaux derivatives $\bm g'_{\rb^o}$ in \cite[Theorem 1.46]{hinze2008optimization}.
\end{proof}

\begin{corollary}\label{cor}
Let the cost $J(\xb,\ub,\rb)$ be
\begin{equation}
J(\xb,\ub,\rb)=\int_0^\tau\inn{\mc Q\xb(t)}{\xb(t)}+\inn{\mc R \ub(t)}{\ub(t)}_{\cs}dt,
\end{equation}
where $\mc Q$ is a positive semi-definite, self-adjoint bounded linear operator on $\ss$, and $\mc{R}$ is a coercive, self-adjoint linear bounded operator on $\cs$. If the minimizer $(\ub^o,\rb^o)$ is in the interior of $U_{ad}\times K_{ad}$, then the following set of equations characterizes $(\xb^o,\pb^o,\ub^o,\rb^o)$:
\begin{equation}\notag
\left\lbrace\begin{array}{ll}
\dot{\xb}^o(t)=\mc{A}\xb^o(t)+\F(\xb^o(t))+\mc{B}(\rb^o)\ub^o(t),& \xb^o(0)=\xb_0,\\[2mm]
\dot{\pb}^o(t)=-(\mc{A}^*+{\F_{\xb^o(t)}^\prime}^*)\pb^o(t)-\mc{Q}\xb^o(t),& \pb^o(\tau)=0,\\[2mm]
\ub^o(t)=-\mc{R}^{-1}\mc{B}^*(\rb^o)\pb^o(t),\\[2mm]
\int_0^{\tau} (\mc{B}'_{\rb^o}\ub^o(t))^*\pb^o(t)\, dt=0.
\end{array}\right.
\end{equation}
\end{corollary}
\begin{proof}
If the optimizer $(\ub^o,\rb^o)$ is in the interior of $U_{ad}\times K_{ad}$, then the optimality conditions of \Cref{thm-optimality} hold if and only if
\begin{flalign}
\bm j_{\ub^o}+\mc{B}^*(\rb^o)\pb^o&=0, \label{op1}\\
\bm j_{\rb^o}+\int_0^{\tau} (\mc{B}'_{\rb^o}\ub^o(t))^*\pb^o(t)\, dt&=0.\label{op2}
\end{flalign}
The derivatives $J'_{\xb^o}(t):L^2(0,\tau;\ss)\to \R$ and $J'_{\ub^o}(t):L^2(0,\tau;\cs)\to \R$ are
\begin{flalign}
J'_{\xb^o}\tilde{\xb}&=\inn{\mc Q \xb^o}{\tilde{\xb}}_{L^2(0,\tau;\ss)},\\
J'_{\ub^o}\tilde{\ub}&=\inn{\mc R \ub^o}{\tilde{\ub}}_{L^2(0,\tau;\cs)}.
\end{flalign}
Identify these functionals with elements $\bm j_{\xb^o}=\mc Q \xb^o(t)$ and $\bm j_{\ub^o}= \mc R \ub^o(t)$, and notice that $\bm j_{\rb^o}=0$. Substituting the derivatives in \eqref{op1} and \eqref{op2} yields the optimality conditions.
\end{proof}

\chgs{For all $\xb_1$ and $\xb_2$ in $D(\A)$ and $t\in (0,\tau)$, let $\Pi(t)$ be the solution to the differential Riccati equation 
\begin{equation}\label{ric}
\begin{cases}
\frac{d}{dt}\inn{\xb_2}{\Pi(t)\xb_1}=-\inn{\xb_2}{\Pi(t)\A\xb_1}-\inn{\A \xb_2}{\Pi(t)\xb_1}\\
\qquad -\inn{\mc Q\xb_2}{\xb_1}+\inn{\Pi(t)\B(\rb)\mc R^{-1}\B^*(\rb)\Pi(t)\xb_2}{\xb_1}, \\
\Pi(\tau)=0.
\end{cases}
\end{equation}

It is well-known, \cite[Chapter 6]{curtain2012introduction} and \cite[Chapter 1]{lasiecka2000control}, that if the system is linear then the adjoint trajectory state $\pb^o(t)$ satisfies 
\begin{equation}
\pb^o(t)=\Pi(t)\xb^o(t).
\end{equation}
As a result, the optimal input and actuator design satisfy in this case
\begin{equation}\notag
\left\lbrace\begin{array}{ll}
\ub^o(t)=-\mc{R}^{-1}\mc{B}^*(\rb^o)\Pi(t)\xb^o(t),\\[2mm]
\int_0^{\tau} (\mc{B}'_{\rb^o}\ub^o(t))^*\Pi(t)\xb^o(t)\, dt=0.
\end{array}\right.
\end{equation}}

\section{Worst Initial Condition}
\chgs{In this section, sets $U_{ad}$ and $K_{ad}$ and numbers $\tau$ and $R_2$ are the same sets and numbers as in the previous section.} 

The worst initial condition maximizes $J(\xb,\ub,\rb)$ over all choices of initial conditions in $B_{\V}(R_2)$ subject to IVP (\ref{eq-IVP}) for a fixed input $\ub\in U_{ad}$ and fixed actuator design $\rb \in K_{ad}$. Formally, define $\mc G(\cdot):\V\to \R$ as $$\mc G(\xb_0)=J(\mc S(\ub,\rb;\xb_0),\ub,\rb),$$
the worst initial condition over $B_{\V}(R_2)$ is the solution to
\begin{equation}
\begin{cases}
\displaystyle\max&\mc G(\xb_0)\\
\text{s.t.}& \xb_0\in B_{\V}(R_2).
\end{cases} \tag{P1} \label{P1}
\end{equation} 

\begin{lemma}\label{lem-x0}
 For every $\ub\in U_{ad}$ and $\rb\in K_{ad}$, the optimization problem (P1) admits a maximizer.
\end{lemma}
\begin{proof}
As in the proof of \Cref{thm-existence optimizer}, define 
\begin{equation}
j:=\sup_{\xb_0\in B_{\V}(R_2)}\mc G(\xb_0).
\end{equation}
Extract a maximizing sequence $\xb^n_0$ in $B_{\V}(R_2)$. The set $B_{\V}(R_2)$ is closed and convex in the reflexive Banach space $\V$, it is therefore weakly closed. This implies that $\xb^n_0$ has a subsequence that converges weakly to some element $\bar{\xb}_0$ in $B_{\V}(R_2)$. Also, according to \Cref{thm-existence} and \Cref{thm-weak}, the solution map is bounded and weakly continuous in $\xb_0$. The cost function is also convex and continuous in $\xb_0$, so it is weakly lower semi-continuous in $\xb_0$. These imply that $\bar{\xb}_0$ solves \eqref{P1}.
\end{proof}

\begin{proposition}\label{prop-x0}
Under assumptions A1-A4, the solution map $\mc S(\ub(t),\rb;\xb_0)$ is G\^ateaux differentiable with respect to $\xb_0\in B_{\V}(R_2)$. Let $\xb(t)=\mc S(\ub(t),\rb,\xb_0)$, the G\^ateaux derivative of $\mc S(\ub(t),\rb;\xb_0)$ at $\xb_0$ in the interior of $B_{\V}(R_2)$ in the direction $\tilde{\xb}_0$ is the mapping $\mc S'_\xb (\ub(t),\rb;\cdot):\V\to\W(0,\tau)$, $\tilde{\xb}_0\mapsto \qb(t)$, where $\qb(t)$ is the strict solution to
\begin{equation}
\begin{cases}
\dot{\qb}(t)=(\A+{\F}_{\xb(t)}^\prime) \qb(t),\\
\qb(0)=\tilde{\xb}_0.
\end{cases}
\end{equation}
\end{proposition}
\begin{proof}
Let the number $\epsilon>0$ be small enough such that $\xb_0+\epsilon\tilde{\xb}_0\in B_{\V}(R_2)$. Define $\xb_\epsilon(t):=\mc S(\ub(t),\rb,\xb_0+\epsilon\tilde{\xb}_0)$, it solves
\begin{equation}
\begin{cases}
\dot{\xb}_\epsilon(t)=\mc{A}\xb_\epsilon(t)+\mc{F}(\xb_\epsilon(t))+\mc{B}(\rb)\ub(t),\quad t>0,\\
 \xb_\epsilon(0)=\xb_0+\epsilon\tilde{\xb}_0.
\end{cases}
\end{equation}
Define $\bm{\be}_{\F}(t)$ as 
\begin{equation}\notag
\bm{\be}_{\F}(t)\coloneqq \frac{1}{\epsilon}\left(\F(\xb(t))-\F(\xb_\epsilon(t))-{\F}_{\xb(t)}^\prime(\xb(t)-\xb_\epsilon(t))\right).
\end{equation}
Let $\be(t)=(\xb(t)-\xb_\epsilon(t))/\epsilon-\qb(t)$, it satisfies 
\begin{equation}\label{edot2}
\begin{cases}
\dot{\be}(t)=(\mc{A}+{\F}_{\xb(t)}^\prime)\be(t)+\bm{\be}_{\F}(t),\\
\be(0)=0.
\end{cases}
\end{equation}
Assumption A4 ensures that as $\epsilon\to 0$
\begin{equation}
\normm{\bm{\be}_{\F}(t)}\to 0, \quad \forall t\in [0,\tau].\label{e(t)3}\\
\end{equation}
The convergence in (\ref{e(t)3}) is uniform; to show this, note that $\xb(t)-\xb_\epsilon(t)$ satisfies 
\begin{equation}\notag
\begin{cases}
\dot{\xb}(t)-\dot{\xb}_\epsilon(t)=\A(\xb(t)-\xb_\epsilon(t))+\F(\xb(t))-\F(\xb_\epsilon(t)), \\
\xb(0)-\xb_\epsilon(0)=\epsilon \tilde{\xb}_0.
\end{cases}
\end{equation}
According to \cite[Proposition 2.2]{clement1993abstract}, there is $d_\tau$ depending only on $\tau$ such that for all $t\in [0,\tau]$
\begin{flalign}\label{eq14}
\norm{\dot{\xb}-\dot{\xb}_\epsilon}{L^2(0,t;\ss)}&+\norm{\A(\xb-\xb_\epsilon)}{L^2(0,t;\ss)}\\
&\le d_\tau \left(\norm{\F(\xb)-\F(\xb_\epsilon)}{L^2(0,t;\ss)} + \epsilon\norm{\tilde{\xb}_0}{\V}\right)\notag 
\end{flalign}
Also, letting $c_\tau$ be the embedding constant of $\W(0,\tau) \hookrightarrow C(0,\tau;\V)$, $\xb -\xb_\epsilon$ satisfies
\begin{flalign}\notag
\norm{\xb -\xb_\epsilon}{C(0,t;\V)}\le c_\tau &\left(\norm{\dot{\xb}-\dot{\xb}_\epsilon}{L^2(0,t;\ss)}\right. \notag \\
&\quad \left. +\norm{\A(\xb-\xb_\epsilon)}{L^2(0,t;\ss)}\right).\label{eq15}
\end{flalign}
\Cref{thm-existence} implies that the states $\xb(t)$ and $\xb_\epsilon(t)$ belong to some bounded set $D$; so let $L_\F$ be the Lipschitz constant $\F(\cdot)$ on $D$. Combining this with inequalities (\ref{eq14}) and (\ref{eq15}) yield
\begin{flalign}
\norm{\xb(t)-\xb_\epsilon(t)}{\V}^2\le& 2 c_\tau^2 d_\tau^2 L_\F^2 \int_0^t \norm{\xb(s)-\xb_\epsilon(s)}{\V}^2ds \notag \\
&+ 2 \epsilon^2 \norm{\tilde{\xb}_0}{\V}^2.
\end{flalign}
Applying Gronwall's lemma to this inequality yields
\begin{equation}\label{eq-14}
\norm{\xb(t)-\xb_\epsilon(t)}{\V}\le \sqrt{2} e^{c_\tau^2 d_\tau^2 L_\F^2} \epsilon \norm{\tilde{\xb}_0}{\V}.
\end{equation}
Take the norm of $\be_{\F}(t)$ in $\ss$, use (\ref{eq-14}), define $$M_{\F}:=\sup\{\|{\F}_{\xb(t)}^\prime\|_{\mc L(\V,\ss)}:t\in [0,\tau]\}.$$
It follows that
\begin{equation}\notag
\normm{\be_{\F}(t)}\le (L_{\F}+M_{\F})\sqrt{2} e^{c_\tau^2 d_\tau^2 L_\F^2} \norm{\tilde{\xb}_0}{\V} <\infty.
\end{equation}
The Bounded Convergence Theorem now ensures that
\begin{equation}\label{eq-15}
\lim_{\epsilon\to 0}\int_0^\tau \norm{\be_{\F}(t)}{\V}^2dt = 0.
\end{equation}
\Cref{lem-estimate} together with (\ref{eq-15}) gives
\begin{equation}
\lim_{\epsilon\to 0} \norm{\be}{\W(0,\tau)}=0.
\end{equation}
This shows that $\mc S(\ub,\rb,\xb_0)$ is G\^ateaux differentiable at $\xb_0$ in the direction $\tilde{\xb}_0$.
\end{proof}

\begin{theorem}\label{thm-worst ini}
Suppose assumptions A1-A4 hold, and identify the derivative $J'_{\xb}$ by element $\bm j_\xb\in \W(0,\tau)^*$. Let $\ub\in U_{ad}$, $\rb \in K_{ad}$, and $\xb=\mc S(\ub,\rb;\xb_0)$. Also, let $\pb(t)$, the adjoint trajectory state, satisfy
\begin{flalign}\label{adj2}
\begin{cases}
\dot{\pb}(t)=-(\A+{\F}_{\xb(t)}^\prime)\pb(t)-\bm j_\xb(t), \quad t\in [0,\tau),\\
\pb(\tau)=0.
\end{cases}
\end{flalign}
If $\xb_0$ is a worst initial condition over $B_{\V}(R_2)$, then, there is a non-negative number $\mu$ such that 
\begin{flalign}\label{opt x0}
\begin{cases}
\mu \left(\normm{\xb_0}-R_2 \right)=0,\\
\pb(0)+\mu \xb_0=0.\\
\end{cases}
\end{flalign}
\end{theorem}
\begin{proof}
Define $f(\xb_0):=\frac{1}{2}(\normm{\xb_0}^2_{\V}-R_2^2).$ Rewrite \eqref{P1} as
\begin{equation}
\begin{cases}
\displaystyle\max&\mc G(\xb_0)\\
\text{s.t.} &f(\xb_0)\le 0.
\end{cases} \label{P2}
\end{equation} 
The constraint $f(\xb_0)\le 0$ satisfies Robinson's regularity condition \cite[Section 1.7.3.2]{hinze2008optimization}. This allows one to apply \cite[Theorem 1.56]{hinze2008optimization}. Let $\lambda$ be a non-negative number, and define the Lagrangian 
\begin{equation}
\mf L(\xb_0,\lambda):=\mc G(\xb_0,\lambda) + \mu f(\xb_0).
\end{equation}
Let $\mf L'_{\xb_0}:\V\to \R$ be the G\^ateaux derivative of $\mf L(\xb_0,\lambda)$ at $\xb_0$. Identify $\mf L'_{\xb_0}$ with an element $l_{\xb_0}\in \V$. Theorem 1.56 of \cite{hinze2008optimization} ensures that the worst initial condition satisfies for all $\tilde{\xb}_0\in \V$ the conditions
\begin{subequations}\label{optim}
\begin{flalign}
f(\xb_0)&\le 0,\\
\mu &\ge 0,\\
\mu f(\xb_0)&=0,\label{slackness}\\
\inn{l_{\xb_0}}{\tilde{\xb}_0-\xb_0}_\V&\ge 0.\label{ineq optim}
\end{flalign}
\end{subequations}

In the following, an explicit expression for $l_{\xb_0}$ will be derived. First, the G\^ateaux derivative of $f(\xb_0)$ at $\xb_0$ along $\tilde{\xb}_0$ is
\begin{equation}
f'_{\xb_0}\tilde{\xb}_0=\inn{\tilde{\xb}_0}{\xb_0}_{\V}.\label{der1}
\end{equation}
This implies that the functional $f'_{\xb_0}:\V \to \R$ can be identified with the element $\xb_0$. The G\^ateaux derivative of $\mc G(\xb_0)$ at $\xb_0$ along $\tilde{\xb}_0$ is derived using the chain rule,
\begin{equation}\label{opt2}
\mc G'_{\xb_0}\tilde{\xb}_0=J'_\xb \mc{S}'_{\xb_0}\tilde{\xb}_0.
\end{equation}
The functionals $\mc G'_{\xb_0}:\V\to \R$ and $J'_\xb:L^2(0,\tau;\ss)\to \R$ can be identified with some elements $\bm g_{\xb_0}\in \V$ and $\bm j_{\xb}\in L^2(0,\tau;\ss)$, respectively. Then, equality \eqref{opt2} implies that
\begin{equation}
\bm g_{\xb_0}=\mc{S}'^*_{\xb_0}\bm j_\xb(t).
\end{equation}
The adjoint operator $\mc{S}'^*_{\xb_0}$ will be derived. Let $\mc{S}'_{\xb_0}\tilde{\xb}_0=\qb(t)$. Consider the inner-product
\begin{flalign}
&\inn{\bm j_\xb}{\mc{S}'_{\xb_0}\tilde{\xb}_0}_{L^2(0,\tau;\ss)}=\int_0^\tau\inn{\bm j_{\xb}(t)}{\qb(t)}dt\notag \\
&=\int_0^\tau\inn{-\dot{\pb}(t)-(\mc{A}+{\F}_{\xb(t)}^\prime)\pb(t)}{\qb(t)}dt.
\end{flalign}
Using \Cref{prop-x0} and applying integration by parts yield
\begin{flalign}
&\inn{\bm j_\xb}{\mc{S}'_{\xb_0}\tilde{\xb}_0}_{L^2(0,\tau;\ss)}= \inn{\pb(0)}{\qb(0)}_{\V}-\inn{\pb(\tau)}{\qb(\tau)}_{\V}\notag \\
&\quad +\int_0^\tau \inn{\pb(t)}{\dot{\qb}(t)-(\mc{A}+{\F}_{\xb(t)}^\prime)\qb(t)}dt\notag\\
& =\inn{\pb(0)}{\tilde{\xb}_0}_{\V}. \label{eq-16}
\end{flalign}
It follows that $\mc{S}'^*_{\xb_0}\bm j_\xb=\pb(0)$, and so
\begin{equation}
\bm g_{\xb_0}=\pb(0).\label{der2}
\end{equation} 
Combining \eqref{der1} and \eqref{der2} yield
\begin{equation}
l_{\xb_0}=\pb(0)+\mu \xb_0.
\end{equation} 
Substituting this in \eqref{optim} yields
\begin{equation}\label{ineq optim 2}
\inn{\pb(0)+\mu \xb_0}{\tilde{\xb}_0-\xb_0}_\V\ge 0, \quad \forall \tilde{\xb}_0\in \V.
\end{equation}
Since $\tilde{\xb}_0\in \V$ is arbitrary, the inequality condition \eqref{ineq optim 2} becomes an equality condition. This together with \eqref{slackness} yields \eqref{opt x0}.
\end{proof}

For linear systems with quadratic cost, the adjoint trajectory state satisfies $\pb^o(t)=\Pi(t)\xb^o(t)$ where $\Pi(t)$ solves \eqref{ric}. Consequently, the optimality condition $\pb^o(0)+\mu \xb_0=0$ becomes
\begin{equation}
\Pi(0)\xb_0=-\mu \xb_0.
\end{equation}
This implies that the worst initial condition is an eigenfunction of the operator $\Pi(0)$.

\section{Kuramoto–Sivashinsky equation}
For every actuator location $r\in (0,1)$, let the function $b(\cdot;r)$ be in $C^1[0,1]$. Consider the controlled Kuramoto–Sivashinsky equation with Dirichlet boundary conditions and initial condition $w_0(\xi)$ on $\xi\in [0,1]$ and some number $\lambda$
\begin{flalign}\notag
\begin{cases}
\begin{aligned}
&\frac{\partial w}{\partial t} +  \frac{\partial^4 w}{\partial \xi^4} +\lambda \frac{\partial^2 w}{\partial \xi^2} + w\frac{\partial w}{\partial \xi} =b(\xi;r)u(t),\ \ &t&> 0,\\[2mm]
&w(0,t) = w(1,t)=0,\quad &t&\ge 0,\\[2mm]
&\frac{\partial w}{\partial \xi}(0,t) = \frac{\partial w}{\partial \xi}(1,t)=0,\quad &t&\ge 0,\\[2mm]
&w(\xi,0)= w_0(\xi), \quad &\xi&\in [0,1].
\end{aligned}
\end{cases}
\end{flalign}

Define the state $\xb(t)\coloneqq w(\cdot,t)$, the state space $\ss\coloneqq L^2(0,1)$.
Let the state operator $\A:D(\A)(\subset \ss)\to \ss$ be
\begin{flalign}
&\A w\coloneqq -w_{\xi\xi\xi\xi}-\lambda w_{\xi\xi},\notag\\
&D(\A)=H^4(0,1)\cap H^2_0(0,1).
\end{flalign}
Also, the control space is $\cs\coloneqq\R$. The actuator design space is $\as\coloneqq\R$. Define ${\V}\coloneqq H^1_0(0,1)$; the nonlinear operator $\F(\cdot):{\V}\to \ss$ and the input operator $\B(\cdot):\as\to \mc{L}(\cs,\ss)$ are defined as 
\begin{flalign}
\F(w)&\coloneqq -ww_\xi,\label{F-KS}\\
\B(r)u&\coloneqq b(\xi,r)u.
\end{flalign}
The state space representation of the model will then be (\ref{eq-IVP}).

The operator $\A:D(\A)\to \ss$ is a self-adjoint operator, is bounded from below, and has compact resolvent. According to Theorem \cite[Theorem 32.1]{sell2013dynamics}, $\A$ generates an analytic semigroup on $\ss$. Since the operator $\A$ is analytic on a Hilbert space, Theorem 4.1 in \cite{dore1993p} ensures that this operator enjoys maximal parabolic regularity. Also, by Rellich-Kondrachov compact embedding theorem \cite[Chapter 6]{adams2003sobolev}, the space $D(\A)$ is compactly embedded in $\ss$. The operator $\A$ is also associated with a form described in A2.
\begin{lemma}
The nonlinear operator $\F(\cdot)$ is G\^ateaux differentiable from ${\V}$ to $\ss$. The G\^ateaux derivative of $\F(\cdot)$ at $w$ in the direction $f$ is ${\F}^\prime_wf=-wf_\xi-w_\xi f$.
\end{lemma}
\begin{proof}
The operator ${\F}^\prime_w$, if exists, needs to satisfy 
\begin{equation}
\lim_{\epsilon\to 0}\norm{\frac{1}{\epsilon}\left(\F(w+\epsilon f)-\F(w)\right)-{\F}^\prime_wf}{L^2}=0.
\end{equation}
Substituting in (\ref{F-KS}), inside the limit becomes
\begin{flalign}
&\norm{\frac{1}{\epsilon}(ww_{\xi}-(w+\epsilon f)(w_{\xi}+\epsilon f_{\xi}))-wf_\xi-w_\xi f}{L^2}\notag \\
&\qquad=\norm{\epsilon f f_{\xi}}{L^2}.
\end{flalign}
Note that $f\in H^1_0(0,1)$. Embedding $H^1_0(0,1)\hookrightarrow C[0,1]$ means that $f$ is a continuous function over $[0,1]$. This implies that $ff_{\xi}$ is in $L^2(0,1)$, thus
\begin{equation}
\lim_{\epsilon\to 0}\epsilon\norm{f f_{\xi}}{L^2}=0.
\end{equation}
The lemma now follows from the uniqueness of G\^ateaux derivative.
\end{proof}
Note that $D_{\A}(1/2,2)=H_0^2(0,1)\hookrightarrow {\V}$ (see \cite[Corollary 4.10]{chandler2015interpolation}). The operator $\F(\cdot):{\V}\to \ss$ is not however weakly continuous, and does not satisfy assumption B1 of \cite{edalatzadehSICON}.

For all functions $f$ and $w$ in $H^1_0(0,1)$ and $g$ in $H^1(0,1)$, the adjoint of ${\F}_w^\prime$ satisfies
\begin{equation}
\inn{f}{{\F_w^\prime}^*g}_{L^2}=\inn{{\F}_w^\prime f}{g}_{L^2}=\int_0^1 (-wf_{\xi}-w_{\xi}f)g d\xi .
\end{equation}
Performing integration by parts yields
\begin{flalign}
\int_0^1 (-wf_{\xi}-w_{\xi}f)g d\xi=-\int_0^1 wg_{\xi}fd\xi.
\end{flalign}
The operator ${\F_w^\prime}^*$ maps $D({\F_w^\prime}^*)=H^1(0,1)$ to $L^2(0,1)$ as follows
\begin{equation}
{\F_w^\prime}^*g=-wg_{\xi}.
\end{equation}
In addition,
\begin{flalign}
\B^*(r)w&=\int_{0}^{1}b(\xi,r)w(\xi)d\xi, &\forall& w\in \V,\quad\\
(\B'_ru)^*f&=u\int_{0}^{1} b_r(\xi;r)f(\xi) d\xi, &\forall& f\in {\V}.\quad        
\end{flalign}
%In section 7.2 of \cite{ahmadi2016}, it is shown that for $\lambda\le (0.62)4\pi$ the system (\ref{eq-ks}) is $D^1$-ISS in $H^1(0,1)$ (see \cite[Def 2.D]{ahmadi2016}). This implies that for every $\xb_0\in \V$ and  that for every 
Also, define 
\begin{equation}
K_{ad}:=\left\{r\in[a,b]:  0<a<b<1\right\}.
\end{equation}

%Section 7.2 of \cite{ahmadi2016} discusses the input-to-state stability (ISS) of KS equation.
Global stability of an uncontrolled KS equation has been studied extensively, see e.g. \cite{al2018linearized,liu2001stability,fantuzzi2016,ahmadi2016}. Theorem 2.1 of \cite{liu2001stability} proves that for $\lambda< 4\pi^2$, the uncontrolled KS equation is globally exponentially stable. Proof of this theorem can be modified to ensure that there is solution to the controlled KS equation over $[0,\tau]$ for all initial conditions in $\V$. The following lemma ensures that for some parameters $\lambda$ there is a solution to the KS equation for all initial conditions and inputs over arbitrary time intervals.

\begin{lemma}\label{lem-KS}
Let $\lambda<4\pi^2$ and $\sigma(\lambda)$ be the smallest eigenvalue of $-\A$. For all initial conditions $w_0\in \V$ and inputs $u\in L^2(0,\tau)$, the strict solution to the KS system satisfies
\begin{equation}\notag
\normm{w(\tau)}^2\le \normm{w_0}^2+\frac{1}{\sigma(\lambda)} \norm{u}{L^2(0,\tau)}^2\max_{\xi\in[0,1]}b^2(\xi;r).
\end{equation} 
\end{lemma}
\begin{proof}
\Cref{thm-existence} ensures that there is a solution $w\in \W(0,\tau)$ over $[0,\tau]$ to the KS system with initial condition $w_0\in \V$ and input $u\in L^2(0,\tau)$. Consider the Lyapunov function 
\begin{equation}
E(t):=\int_0^1 w^2(\xi,t)\, d\xi.
\end{equation}
Since $w\in W^{1,2}(0,\tau;\ss)$, the function $E(t)$ is differentiable. Taking the derivative of $E(t)$ and applying \cite[Lemma 3.1]{liu2001stability} yield
\begin{equation}\label{eq1}
\dot{E}(t)\le -2\sigma(\lambda) E(t)+2\int_0^1w(\xi,t)b(\xi;r)u(t)d\xi.
\end{equation}
Apply Young's inequality to the integral term, for every $\epsilon>0$, 
\begin{equation}
\dot{E}(t)\le (-2\sigma(\lambda)+\epsilon) E(t)+\frac{1}{\epsilon}\int_0^1b^2(\xi;r)u^2(t)d\xi.
\end{equation}
Let $\epsilon= \sigma(\lambda)$. Taking an integral over $[0,\tau]$ yields the desired inequality in the lemma.
\end{proof}

Since the KS system satisfies assumptions A1-A5, \Cref{cor} can be applied to obtain the optimality conditions. The cost function to be optimized is
\begin{equation}
J(\xb,\ub,\rb)=\int_0^\tau\int_0^1w^2(\xi,t)d\xi dt+\int_0^\tau u^2(t)dt.
\end{equation}
Letting $\pb(t)=f(\cdot,t)$, the optimizer $(u^o,r^o,w^o,f^o)$ with initial condition $w_0(\xi)\in H^1_0(0,1)$ satisfies
{\small\begin{flalign}\label{optimizers KSE}
&\begin{cases}\notag
\begin{aligned}
&\p{w^o}{t}{}+\p{w^o}{\xi}{4}+\lambda \p{w^o}{\xi}{2}+w^o\p{w^o}{\xi}{}=b(\xi;r^o)u^o(t), &t&>0\\[2mm]
&w^o(0,t)=w^o(1,t)=0, &t&>0\\[2mm]
&\p{w^o}{\xi}{}(0,t)=\p{w^o}{\xi}{}(1,t)=0, &t&>0\\[2mm]
&w^o(\xi,0)=w_0(\xi),&&
\end{aligned}
\end{cases}\\
&\begin{cases}\notag
\begin{aligned}
&\p{f^o}{t}{}-\p{f^o}{\xi}{4}-\lambda \p{f^o}{\xi}{2}-w^o\p{f^o}{\xi}{}=-w^o(\xi,t),  &t&>0\\[2mm]
&f^o(0,t)=f^o(1,t)=0, &t&>0\\[2mm]
&\p{f^o}{\xi}{}(0,t)=\p{f^o}{\xi}{}(1,t)=0, &t&>0\\[2mm]
&f^o(\xi,\tau)=0,&&
\end{aligned}
\end{cases}\\
&\begin{cases}\notag
\begin{aligned}
&u^o(t)=-\int_0^{1}b(\xi;r^o)f^o(\xi,t)\, d\xi,  \quad t>0,\\[2mm]
&\int_0^{\tau}\int_{0}^{1}u^o(t)b_{r}(\xi;r^o)f^o(\xi,t)\,d\xi dt=0.
\end{aligned}
\end{cases}
\end{flalign}}%%
 
The worst initial condition over a unit ball satisfies 
\begin{equation}
\begin{cases}
\mu\left( \norm{{w_0}}{H_0^1(0,1)}- 1 \right)=0,\\[2mm]
f^o(\xi,0)+\mu{w}_0(\xi)= 0.
\end{cases}
\end{equation}

\section{Nonlinear Diffusion}
Consider the transfer of heat in a bounded, open, connected set $\Omega\subset \mathbb{R}^2$. It is assumed that $\Omega$ has a Lipschitz boundary  separated into $\partial \Omega=\overline{\Gamma_0\cup\Gamma_1}$ where  $\Gamma_0\cap\Gamma_1=\emptyset$ and $\Gamma_0\neq \emptyset$. Denote by $\nu$  the unit outward normal vector field on $\partial \Omega$.
The  class of nonlinear heat transfer models is, for actuator shape $r \in C^1 (\overline{\Omega} ) ,$
\begin{equation}\small\notag
\begin{cases}
\begin{aligned}
&\p{w}{t}{}(\xi,t)=\\
&\qquad\Delta w(\xi,t)+F(w(\xi,t))+r(\xi)u(t),  &(\xi,t)&\in \Omega\times (0,\tau],\\
&w(\xi,t)=0,  &(\xi,t)&\in \Gamma_0\times [0,\tau],\\
&\p{w}{\nu}{}(\xi,t)=0,  &(\xi,t)&\in \Gamma_1\times [0,\tau],\\
&w(\xi,0)=w_0(\xi), &\xi &\in \Omega.
\end{aligned}
\end{cases}
\end{equation}

Defining $\as=L^2(\Omega),$ 
a set of  admissible actuator shapes is
$$K_{ad}  = \{ r \in C^1 (\overline{\Omega} ) :  \norm{r}{C^1 (\overline{\Omega})}\leq 1 \} . $$ 
The set $K_{ad}$ is compact in $\as$ with respect to the norm topology \cite[Chapter 6]{adams2003sobolev}. 

Let $\ss:= L^2(\Omega)$, $\cs\coloneqq\R$, and the state $\xb(t):=w(\cdot,t)$. The operator $\A:D(\A)\to \ss$ is defined as
\begin{subequations}
\begin{gather}
\A w=\Delta w,\\
D(\A)=\left\{w\in H^2(\Omega)\cap H^1_{\Gamma_0}:  \p{w}{\nu}{}=0 \text{ on }\Gamma_1\right\}.
\end{gather}
\end{subequations}
The operator $\A$  self-adjoint, non-negative and has compact resolvent. Thus, it generates an analytic semi-group on the Hilbert space $L^2(\Omega)$ \cite[Theorem 32.1]{sell2013dynamics}, and  has maximal $L^p$ regularity.% by \cite[Thm. 4.1]{dore1993p}.

\begin{comment}
Different nonlinear functions $F(\cdot)$ are used in different applications. If $F(\zeta)$ has sub-linear growth, i.e., there is a positive number $a$ such that 
\begin{equation}
|F(\zeta)|\le a(1+|\zeta|),
\end{equation}
then, $\V=\ss$. However, in most applications, $F(\zeta)$ has super-linear growth, so $ \V \subsetneq \ss$. 
\end{comment}
Define $\V=H^1_{\Gamma_0}(\Omega)$ and assume that  the nonlinear operator $\F(\cdot) : \V \to \ss .$
The proof of the following lemma is  the same as that of  \cite[Lemma 7.1.1]{edalatzadehSICON}.
\begin{lemma}\label{lem-heat-F}
Let $\V=H^1_{\Gamma_0}(\Omega)$. Assume that
\begin{enumerate}
\item $F(\zeta)$ is twice continuously differentiable over $\mathbb{R}$; denote its derivatives by $\F^\prime(\zeta)$ and $F^{\prime\prime} (\zeta)$;
\item \label{C2}there are numbers $a_0>0$ and $b>1/2$ such that $|F^{\prime \prime}(\zeta)|\le a_0(1+|\zeta|^b)$.
\end{enumerate}
Then $\F(\cdot)$ is G\^ateaux differentiable from $\V$ to $\ss$. The G\^ateaux derivative of $\F(\cdot)$ at $w(\xi)$ in the direction $f(\xi)$ is ${\F^\prime}_wf=F^\prime(w)f$.
\end{lemma}

It is straightforward to show that the operator $\F_w^\prime:\V(\subset \ss)\to \ss$ is self-adjoint, i.e.,
\begin{equation}
\inn{{\F_w^\prime}^*g}{f}=\inn{g}{{\F_w^\prime}f}, \quad \forall f,g \in \V. 
\end{equation}

Define $\cs =\mathbb{R}$ and the input operator $\B(r)\in \mc{L}(\cs ,\ss)$ maps $u$ to $r(\xi)u$.
Also, for all $f$ in $\ss$
\begin{flalign}
\B^*(r)f&=\int_{\Omega}r(\xi)f(\xi)d \xi,\\
(\B'_ru)^*f&=uf. 
\end{flalign}

For every initial condition in $\V$, a strict solution over $[0,\tau]$ to the nonlinear heat equation is not guaranteed. The following lemma states a condition under which there is a solution to the diffusion equation for all initial conditions and inputs over arbitrary time intervals.
\begin{lemma}\label{lem-heat}
If the function $F(\zeta)$ satisfies $\zeta F(\zeta)\le 0$ for all $\zeta\in \mathbb{R}$, then there is $c_\Omega>0$ such that the strict solution to the nonlinear heat equation satisfies
\begin{equation}\notag
\normm{w(\tau)}^2\le \normm{w_0}^2+\frac{4}{c_\Omega} \norm{u}{L^2(0,\tau)}^2\norm{r}{\as}^2.
\end{equation}
\end{lemma}
\begin{proof}
Theorem 1 in \cite{mazenc2011strict} proves that the nonlinear equation in one spatial dimension is input-to-state stable. This lemma extends \cite[Theorem 1]{mazenc2011strict} to two-spatial dimension. Using the same idea of proof, consider the Lyapunov function
\begin{equation}
E(t):=\int_{\Omega}w^2(\xi,t) \, d\xi.         
\end{equation}
The function $E(t)$ is differentiable since $w\in W^{1,2}(0,\tau;\ss)$. Take the derivative of this function, substitute for $\dot{w}(\xi,t)$ from the heat equation, and perform integration by parts as follows
\begin{flalign}
\dot{E}(t)=&2\int_{\Omega}w(\xi,t)\left( \Delta w(\xi,t)+F(w(\xi,t))+r(\xi)u(t) \right)\, d\xi\notag\\
=&2\int_{\Gamma}w(\xi,t)\p{w}{\nu}{}(\xi,t)d\xi-2\int_{\Omega}\left(\nabla w(\xi,t)\right)^2\, d\xi\notag\\
&+2\int_{\Omega}w(\xi,t)\left( F(w(\xi,t))+r(\xi)u(t) \right)\, d\xi.
\end{flalign} 
Apply the boundary conditions. Use Poincar\'e inequality and let $c_\Omega$ be its constant. Also, use Young's inequality for all $\epsilon>0$ 
\begin{equation}\label{eq2}
\dot{E}(t)\le -2\left(c_{\Omega}-\epsilon\right)E(t)+\frac{2}{\epsilon} u^2(t)\norm{r}{2}^2.
\end{equation}
Set $\epsilon=c_\Omega/2$. Taking the integral over $[0,\tau]$ of (\ref{eq2}) then yields the desired inequality.
\end{proof}

The nonlinear heat equation satisfies assumptions A1-A5, and thus, \Cref{cor} can be applied to obtain the optimality conditions. The cost function to be optimized is
\begin{equation}
J(\xb,\ub,\rb)=\int_0^\tau  \int_{\Omega} w^2(\xi,t)d\xi dt +\int_0^\tau  u^2(t)dt.
\end{equation} 
Letting $\pb(t)=f(\cdot,t)$, The optimizer $(u^o,r^o,w^o,f^o)$ with initial condition $w_0\in H^1_{\Gamma_0}(\Omega)$   satisfies
{\small\begin{flalign}\label{optimizers heat}
&\begin{cases}\notag
\begin{aligned}
&\p{w^o}{t}{}(\xi,t)=\\
&\qquad\Delta w^o(\xi,t)+F(w^o(\xi,t))+r^o(\xi)u^o(t),  &(\xi,t)&\in \Omega\times (0,\tau],\\[1mm]
&w^o(\xi,t)=0,  &(\xi,t)&\in \Gamma_0\times [0,\tau],\\[1mm]
&\p{w^o}{\nu}{}(\xi,t)=0,  &(\xi,t)&\in \Gamma_1\times [0,\tau],\\[1mm]
&w^o(\xi,0)=w_0(\xi), &\xi &\in \Omega.
\end{aligned}
\end{cases}\\
&\begin{cases}\notag
\begin{aligned}
&\p{f^o}{t}{}(\xi,t)=-\Delta f^o(\xi,t)\\
&\quad -F^\prime(w(\xi,t))f^o(\xi,t)-w^o(\xi,t),  &(\xi,t)&\in \Omega\times (0,\tau],\\[1mm]
&f^o(\xi,t)=0,  &(\xi,t)&\in \Gamma_0\times [0,\tau],\\[1mm]
&\p{f^o}{\nu}{}(\xi,t)=0,  &(\xi,t)&\in \Gamma_1\times [0,\tau],\\[1mm]
&f^o(\xi,\tau)=0, &\xi &\in \Omega,
\end{aligned}
\end{cases}\\
&\begin{cases}\notag
\begin{aligned}
&u^o(t)=-\int_{\Omega}r^o(\xi)f^o(\xi,t)\, d\xi,  &t&\in [0,\tau],\\[2mm]
&\int_0^{\tau}u^o(t)f^o(\xi,t) dt=0, &\xi&\in \Omega.
\end{aligned}
\end{cases}
\end{flalign}}%%

The worst initial condition over a unit ball satisfies
\begin{equation}
\begin{cases}
\mu\left(\norm{{w}_0}{H^1_{\Gamma_0}(\Omega)}- 1\right)=0,\\[2mm]
f^o(\xi,0)+\mu w_0(\xi)= 0.
\end{cases}
\end{equation}

\section{Conclusion}
Optimal actuator design for quasi-linear infinite-dimensional systems with a parabolic linear part was considered in this paper. It was shown that the existence of an optimal control together with an optimal actuator design is guaranteed under natural assumptions. With additional assumptions of differentiability,  first-order necessary optimality conditions were obtained. The theory was illustrated by application to the  Kuramoto-Sivashinsky (KS) equation and nonlinear heat equations. 

Current work is concerned with developing numerical methods for solution of the optimality equations. Extension of these problems to situations  where the input operator is not bounded on the state space is also of interest.

\bibliography{library}

\begin{thebibliography}{10}
\expandafter\ifx\csname url\endcsname\relax
  \def\url#1{\texttt{#1}}\fi
\expandafter\ifx\csname urlprefix\endcsname\relax\def\urlprefix{URL }\fi
\expandafter\ifx\csname href\endcsname\relax
  \def\href#1#2{#2} \def\path#1{#1}\fi

\bibitem{morris2015comparison}
K.~Morris, S.~Yang, {Comparison of actuator placement criteria for control of
  structures}, Journal of Sound and Vibration 353 (2015) 1--18.

\bibitem{frecker2003recent}
M.~I. Frecker, {Recent advances in optimization of smart structures and
  actuators}, Journal of Intelligent Material Systems and Structures 14~(4-5)
  (2003) 207--216.

\bibitem{van2001review}
M.~{Van De Wal}, B.~{De Jager}, {A review of methods for input/output
  selection}, Automatica 37~(4) (2001) 487--510.

\bibitem{morris2011linear}
K.~Morris, {Linear-quadratic optimal actuator location}, IEEE Transactions on
  Automatic Control 56~(1) (2011) 113--124.

\bibitem{DM2013}
K.~A. Morris, M.~A. Demetriou, S.~D. Yang, {Using ${H}_2$-control performance
  metrics for infinite-dimensional systems}, IEEE Transactions on Automatic
  Control 60~(2) (2015) 450--462.

\bibitem{kasinathan2013h}
D.~Kasinathan, K.~Morris, {$\mathbb{H}_\infty$-optimal actuator location}, IEEE
  Transactions on Automatic Control 58~(10) (2013) 2522--2535.

\bibitem{martinez2000}
A.~Mart{\'{i}}nez, C.~Rodr{\'{i}}guez, M.~E. V{\'{a}}zquez-M{\'{e}}ndez,
  {Theoretical and Numerical Analysis of an Optimal Control Problem Related to
  Wastewater Treatment}, SIAM Journal on Control and Optimization 38~(5) (2000)
  1534--1553.

\bibitem{unger2001}
A.~Unger, F.~Tr{\"{o}}ltzsch, {Fast solution of optimal control problems in the
  selective cooling of steel}, ZAMM - Journal of Applied Mathematics and
  Mechanics / Zeitschrift f{\"{u}}r Angewandte Mathematik und Mechanik 81~(7)
  (2001) 447--456.

\bibitem{li2003}
C.~Li, E.~Feng, J.~Liu, {Optimal control of systems of parabolic PDEs in
  exploitation of oil}, Journal of Applied Mathematics and Computing 13~(1)
  (2003) 247.

\bibitem{boldrini2009}
J.~L. Boldrini, B.~M.~C. Caretta, E.~Fern{\'{a}}ndez-Cara, {Some optimal
  control problems for a two-phase field model of solidification}, Revista
  Matem{\'{a}}tica Complutense 23~(1) (2009) 49.

\bibitem{homberg2010optimal}
D.~H{\"{o}}mberg, C.~Meyer, J.~Rehberg, W.~Ring, D.~H. Omberg, {Optimal control
  for the thermistor problem}, SIAM Journal on Control and Optimization 48~(5)
  (2010) 3449--3481.

\bibitem{buchholz2013}
R.~Buchholz, H.~Engel, E.~Kammann, F.~Tr{\"{o}}ltzsch, {On the optimal control
  of the Schl{\"{o}}gl-model}, Computational Optimization and Applications
  56~(1) (2013) 153--185.

\bibitem{casas2013}
E.~Casas, C.~Ryll, F.~Tr{\"{o}}ltzsch, {Sparse optimal control of the
  Schl{\"{o}}gl and FitzHugh-Nagumo systems}, Computational Methods in Applied
  Mathematics 13~(4) (2013) 415--442.

\bibitem{edalatzadeh2016boundary}
M.~S. Edalatzadeh, A.~Alasty, {Boundary exponential stabilization of
  non-classical micro/nano beams subjected to nonlinear distributed forces},
  Applied Mathematical Modelling 40~(3) (2016) 2223--2241.

\bibitem{reyes2016}
J.~C. de~los Reyes, R.~Herzog, C.~Meyer, {Optimal control of static
  elastoplasticity in primal formulation}, SIAM Journal on Control and
  Optimization 54~(6) (2016) 3016--3039.

\bibitem{yousept2017optimal}
I.~Yousept, {Optimal control of non-smooth hyperbolic evolution Maxwell
  equations in type-II superconductivity}, SIAM Journal on Control and
  Optimization 55~(4) (2017) 2305--2332.

\bibitem{fleig2017}
A.~Fleig, R.~Guglielmi, {Optimal control of the Fokker--Planck equation with
  space-dependent controls}, Journal of Optimization Theory and Applications
  174~(2) (2017) 408--427.

\bibitem{ciaramella2016}
G.~Ciaramella, A.~Borzi, {Quantum optimal control problems with a sparsity cost
  functional}, Numerical Functional Analysis and Optimization 37~(8) (2016)
  938--965.

\bibitem{hintermuller2017optimal}
M.~Hinterm\"uller, T.~Keil, D.~Wegner, {Optimal control of a semidiscrete
  Cahn--Hilliard--Navier--Stokes system with nonmatched fluid densities}, SIAM
  Journal on Control and Optimization 55~(3) (2017) 1954--1989.

\bibitem{merger2017optimal}
J.~Merger, A.~Borzi, R.~Herzog, {Optimal control of a system of
  reaction--diffusion equations modeling the wine fermentation process},
  Optimal Control Applications and Methods 38~(1) (2017) 112--132.

\bibitem{sprengel2018investigation}
M.~Sprengel, G.~Ciaramella, A.~Borz{\`\i}, Investigation of optimal control
  problems governed by a time-dependent kohn-sham model, Journal of Dynamical
  and Control Systems 24~(4) (2018) 657--679.

\bibitem{kimmerle2018optimal}
S.-J. Kimmerle, M.~Gerdts, R.~Herzog, {Optimal control of an elastic
  crane-trolley-load system-a case study for optimal control of coupled ODE-PDE
  systems}, Mathematical and Computer Modelling of Dynamical Systems 24~(2)
  (2018) 182--206.

\bibitem{edalatzadeh2019stability}
M.~S. Edalatzadeh, K.~A. Morris, {Stability and well-posedness of a nonlinear
  railway track model}, IEEE Control Systems Letters 3~(1) (2019) 162--167.

\bibitem{hinze2008optimization}
M.~Hinze, R.~Pinnau, M.~Ulbrich, S.~Ulbrich, {Optimization with PDE
  constraints}, Vol.~23, Springer Science \& Business Media, 2008.

\bibitem{leugering2012constrained}
G.~Leugering, S.~Engell, A.~Griewank, M.~Hinze, R.~Rannacher, V.~Schulz,
  M.~Ulbrich, S.~Ulbrich, {Constrained optimization and optimal control for
  partial differential equations}, Vol. 160, Springer Science \& Business
  Media, 2012.

\bibitem{troltzsch2010optimal}
F.~Tr{\"{o}}ltzsch, {Optimal control of partial differential equations: theory,
  methods, and applications}, Graduate studies in mathematics, American
  Mathematical Society, 2010.

\bibitem{bergounioux2003structure}
M.~Bergounioux, K.~Kunisch, {On the structure of Lagrange multipliers for
  state-constrained optimal control problems}, Systems \& control letters
  48~(3-4) (2003) 169--176.

\bibitem{casas1997pontryagin}
E.~Casas, {Pontryagin's principle for state-constrained boundary control
  problems of semilinear parabolic equations}, SIAM Journal on Control and
  Optimization 35~(4) (1997) 1297--1327.

\bibitem{raymond1999hamiltonian}
J.~P. Raymond, H.~Zidani, {Hamiltonian Pontryagin's principles for control
  problems governed by semilinear parabolic equations}, Applied mathematics \&
  optimization 39~(2) (1999) 143--177.

\bibitem{meyer2017optimal}
C.~Meyer, L.~M. Susu, {Optimal control of nonsmooth, semilinear parabolic
  equations}, SIAM Journal on Control and Optimization 55~(4) (2017)
  2206--2234.

\bibitem{antoniades2001integrating}
C.~Antoniades, P.~D. Christofides, {Integrating nonlinear output feedback
  control and optimal actuator/sensor placement for transport-reaction
  processes}, Chemical Engineering Science 56~(15) (2001) 4517--4535.

\bibitem{lou2003optimal}
Y.~Lou, P.~D. Christofides, {Optimal actuator/sensor placement for nonlinear
  control of the Kuramoto-Sivashinsky equation}, IEEE Transactions on Control
  Systems Technology 11~(5) (2003) 737--745.

\bibitem{armaou2008robust}
A.~Armaou, M.~A. Demetriou, {Robust detection and accommodation of incipient
  component and actuator faults in nonlinear distributed processes}, AIChE
  journal 54~(10) (2008) 2651--2662.

\bibitem{moon2006finite}
S.~H. Moon, {Finite element analysis and design of control system with feedback
  output using piezoelectric sensor/actuator for panel flutter suppression},
  Finite Elements in Analysis and Design 42~(12) (2006) 1071--1078.

\bibitem{saviz2015optimal}
M.~R. Saviz, {An optimal approach to active damping of nonlinear vibrations in
  composite plates using piezoelectric patches}, Smart Materials and Structures
  24~(11) (2015) 115024.

\bibitem{kuramoto1975formation}
Y.~Kuramoto, T.~Tsuzuki, {On the formation of dissipative structures in
  reaction-diffusion systems: Reductive perturbation approach}, Progress of
  Theoretical Physics 54~(3) (1975) 687--699.

\bibitem{sivashinsky1977nonlinear}
G.~Sivashinsky, Nonlinear analysis of hydrodynamic instability in laminar
  flames---i. derivation of basic equations, Acta astronautica 4 (1977)
  1177--1206.

\bibitem{craster2009dynamics}
R.~V. Craster, O.~K. Matar, {Dynamics and stability of thin liquid films},
  Reviews of modern physics 81~(3) (2009) 1131.

\bibitem{novick1987interfacial}
A.~Novick-Cohen, {Interfacial instabilities in directional solidification of
  dilute binary alloys: The Kuramoto-Sivashinsky equation}, Physica D:
  Nonlinear Phenomena 26~(1-3) (1987) 403--410.

\bibitem{laquey1975}
R.~E. LaQuey, S.~M. Mahajan, P.~H. Rutherford, W.~M. Tang, {Nonlinear
  saturation of the trapped-ion mode}, Phys. Rev. Lett. 34~(7) (1975) 391--394.

\bibitem{bena1993}
I.~Bena, C.~Misbah, A.~Valance, {Nonlinear evolution of a terrace edge during
  step-flow growth}, Phys. Rev. B 47~(12) (1993) 7408--7419.

\bibitem{christofides2000}
P.~D. Christofides, A.~Armaou, {Global stabilization of the
  Kuramoto-Sivashinsky equation via distributed output feedback control},
  Systems and Control Letters 39~(4) (2000) 283--294.

\bibitem{gomes2017}
S.~N. Gomes, D.~T. Papageorgiou, G.~A. Pavliotis, {Stabilizing non-trivial
  solutions of the generalized Kuramoto-Sivashinsky equation using feedback and
  optimal control}, IMA Journal of Applied Mathematics (Institute of
  Mathematics and Its Applications) 82~(1) (2017) 158--194.

\bibitem{tomlin2019point}
R.~J. Tomlin, S.~N. Gomes, Point-actuated feedback control of multidimensional
  interfaces, arXiv preprint arXiv:1901.09223 (2019).

\bibitem{cerpa2010}
E.~Cerpa, {Null controllability and stabilization of the linear
  Kuramoto-Sivashinsky equation}, Communications on Pure and Applied Analysis
  9~(1) (2010) 91--102.

\bibitem{cerpa2011}
E.~Cerpa, A.~Mercado, {Local exact controllability to the trajectories of the
  1-D Kuramoto-Sivashinsky equation}, Journal of Differential Equations 250~(4)
  (2011) 2024--2044.

\bibitem{sun2010}
B.~Sun, {Maximum principle for optimal boundary control of the
  Kuramoto-Sivashinsky equation}, Journal of the Franklin Institute 347~(2)
  (2010) 467--482.

\bibitem{gao2016}
P.~Gao, {Optimal distributed control of the Kuramoto-Sivashinsky equation with
  pointwise state and mixed control-state constraints}, IMA Journal of
  Mathematical Control and Information 33~(3) (2016) 791--811.

\bibitem{liu2001stability}
W.-j. Liu, M.~Krstic, {Stability enhancement by boundary control in the
  Kuramoto–Sivashinsky equation}, Nonlinear Analysis: Theory, Methods \&
  Applications 46~(4) (2001) 485 -- 507.

\bibitem{al2018linearized}
R.~{Al Jamal}, K.~Morris, {Linearized stability of partial differential
  equations with application to stabilization of the Kuramoto--Sivashinsky
  equation}, SIAM Journal on Control and Optimization 56~(1) (2018) 120--147.

\bibitem{lunardi2012analytic}
A.~Lunardi, {Analytic semigroups and optimal regularity in parabolic problems},
  Springer Science \& Business Media, 2012.

\bibitem{bensoussan2015book}
A.~Bensoussan, G.~{Da Prato}, M.~C. Delfour, S.~K. Mitter, {Representation and
  Control of Infinite Dimensional Systems}, Vol.~1, 2015.

\bibitem{dore1993p}
G.~Dore, {Lp regularity for abstract differential equations}, in: Functional
  Analysis and Related Topics, 1991, Springer, 1993, pp. 25--38.

\bibitem{sell2013dynamics}
G.~R. Sell, Y.~You, {Dynamics of evolutionary equations}, Vol. 143, Springer
  Science \& Business Media, 2013.

\bibitem{clement1993abstract}
P.~Cl{\'{e}}ment, S.~Li, {Abstract parabolic quasilinear equations and
  applications to a groundwater flow problem}, Advances in mathematical
  sciences and applications 3 (1993) 17--32.

\bibitem{amann2000compact}
H.~Amann, {Compact embeddings of vector-valued sobolev and Besov spaces},
  Glasnik matemati{\v{c}}ki 35~(55) (2000) 161--177.

\bibitem{edalatzadehSICON}
M.~S. Edalatzadeh, K.~A. Morris, Optimal actuator design for semilinear
  systems, SIAM Journal on Control and Optimization 57~(4) (2019) 2992--3020.

\bibitem{lang2012real}
S.~Lang, Real and functional analysis, Vol. 142, Springer Science \& Business
  Media, 2012.

\bibitem{dier2015}
D.~Dier, {Non-autonomous maximal regularity for forms of bounded variation},
  Journal of Mathematical Analysis and Applications 425~(1) (2015) 33--54.

\bibitem{wouk1979course}
A.~Wouk, {A course of applied functional analysis} (1979).

\bibitem{curtain2012introduction}
R.~F. Curtain, H.~Zwart, {An introduction to infinite-dimensional linear
  systems theory}, Texts in Applied Mathematics, Springer New York, 2012.

\bibitem{lasiecka2000control}
I.~Lasiecka, R.~Triggiani, Control theory for partial differential equations:
  continuous and approximation theories, Vol.~1, Cambridge University Press
  Cambridge, 2000.

\bibitem{adams2003sobolev}
R.~A. Adams, J.~J.~F. Fournier, {Sobolev spaces}, Pure and Applied Mathematics,
  Elsevier Science, 2003.

\bibitem{chandler2015interpolation}
S.~N. Chandler-Wilde, D.~P. Hewett, A.~Moiola, {Interpolation of Hilbert and
  Sobolev spaces: quantitative estimates and counterexamples}, Mathematika
  61~(2) (2015) 414--443.
\newblock \href {http://arxiv.org/abs/1404.3599} {\path{arXiv:1404.3599}}.

\bibitem{fantuzzi2016}
G.~Fantuzzi, A.~Wynn, {Semidefinite relaxation of a class of quadratic integral
  inequalities}, in: Decision and Control (CDC), 2016 IEEE 55th Conference on,
  IEEE, 2016, pp. 6192--6197.

\bibitem{ahmadi2016}
M.~Ahmadi, G.~Valmorbida, A.~Papachristodoulou, {Dissipation inequalities for
  the analysis of a class of PDEs}, Automatica 66 (2016) 163--171.

\bibitem{mazenc2011strict}
F.~Mazenc, C.~Prieur, {Strict Lyapunov functions for semilinear parabolic
  partial differential equations}, Mathematical Control and Related Fields
  1~(2) (2011) 231--250.

\end{thebibliography}
\bibliographystyle{elsarticle-num}
\end{document}